\documentclass[11pt]{amsart}

\usepackage{amsmath, amssymb, amsthm, mathrsfs}
\usepackage{hyperref}
\usepackage{geometry}
\usepackage{mathtools}
\usepackage{xcolor}
\DeclarePairedDelimiter{\tnorm}{\vert\!\vert\!\vert}{\vert\!\vert\!\vert}
\newtheorem{theorem}{Theorem}
\newtheorem{proposition}{Proposition}
\newtheorem{lemma}{Lemma}
\newtheorem{corollary}{Corollary}
\theoremstyle{definition}
\newtheorem{definition}{Definition}
\theoremstyle{remark}
\newtheorem{remark}{Remark}
 \newtheorem*{ack*}{Acknowledgements}
 \newtheorem*{fund*}{Funding}
 \newtheorem*{coi*}{Conflict of interest}
\newtheorem*{ai*}{AI-use Disclosure}

\DeclareMathOperator{\sing}{sing}

\newcommand{\C}{\mathbb{C}}
\newcommand{\R}{\mathbb{R}}

\newcommand{\vol}{\mathrm{vol}}

\newcommand{\dist}{\mathrm{dist}}

\usepackage[
  backend=bibtex,
  style=numeric, % or style=authoryear, etc.
  sorting=nyt,
  maxnames=50,
  doi=false,isbn=false,url=false,eprint=false,
  giveninits=true
]{biblatex}
\title[Special Lagrangians with isolated singularities]{Special Lagrangians with multiple isolated singularities}
\author{Bryan Dimler and Filippo Gaia}
\date{\today}
\address{Department of Mathematics \\ University of Notre Dame \\ Notre Dame, IN 46556 \\ USA}
\email{bdimler@nd.edu}
\address{Department of Mathematics\\Stanford University\\Stanford, CA 94305\\USA}
\email{fgaia@stanford.edu}

\begin{document}
\begin{abstract}
We extend the Caffarelli–Hardt–Simon perturbation argument for truncated regular minimal cones to the special Lagrangian setting and prove a bridge principle for regular special Lagrangian cones in the spirit of Nathan Smale. Our bridge principle yields a general existence theorem for conically singular special Lagrangian submanifolds with prescribed regular tangent cones: for any finite list of such cones in $\mathbb C^m$ having the same Lagrangian angle and suitably arranged, there exists a connected special Lagrangian submanifold with boundary and isolated conical singularities whose tangent cones at its singularities are precisely the prescribed cones. When applied to the graphical special Lagrangian cone in $\mathbb{C}^5$ recently discovered by Bhattacharya, Orriols, and Skorobogatova, we obtain infinite families of $C^{1,1}$ viscosity solutions to the special Lagrangian equation having any finite number of isolated singularities.
\end{abstract}

\maketitle

\section{Introduction}
In this paper, we show that by gluing together any finite list of regular special Lagrangian cones with the same Lagrangian angle along flat Lagrangian strips (the \emph{$\varepsilon$-bridges}), one can produce examples of connected special Lagrangian submanifolds with boundary in $\C^m$ having any finite number of isolated conical singularities and prescribed tangent cones at their singularities.
More precisely, we prove the following general result, which may be viewed as a special Lagrangian analogue of Nathan Smale's bridge principle for minimal cones (see \cite{Sm89, Sm93}).

\begin{theorem}\label{thm: main-theorem}
    Let $p_1,\ldots,p_N\in \C^m$ be distinct points, and let $C^1,\ldots,C^N$ be truncated special Lagrangian regular cones centered at the points $p_1,\ldots,p_N$. Assume that, for any $i\in\{1,\ldots,N-1\}$, there is an affine Lagrangian plane $T^i$ in $\C^m$ such that both $C^i$ and $C^{i+1}$ are tangent to $T^i$ along a ray.
    Then there exists a connected special Lagrangian submanifold (with boundary) in $\C^m$ having $N$ isolated conical singularities modeled on the cones $C^1,\ldots,C^N$. 
\end{theorem}
\begin{remark}
The theorem applies, in particular, to the case of multiple regular special Lagrangian cones sharing a common tangent plane. We further note that the assumption implies that the cones $C^1,\ldots,C^N$ all have the same Lagrangian angle. Likewise, since any finite list of regular special Lagrangian cones having the same Lagrangian angle can be arranged in $\C^m$ to lie tangent to a common plane, Theorem \ref{thm: main-theorem} can be seen as a general existence theorem for conically singular special Lagrangians with prescribed regular tangent cones. Note also that the result remains true for slightly more general configurations of cones, and with more general bridges (see Remark \ref{rmk: more-general-examples}).\\
We also remark that if the truncated cones $C^1,..., C^N$ are disjoint, then the special Lagrangian submanifold constructed in Theorem \ref{thm: main-theorem} can be chosen to be embedded (see Remark \ref{rmk: embeddedness}).
\end{remark}

Theorem \ref{thm: main-theorem} allows one to generatively produce conically singular special Lagrangians with finitely many isolated conical singularities from regular special Lagrangian cones. While for $m\leq2$ there are no non-flat special Lagrangian cones, for $m\geq 3$ many examples of such cones are known (e.g. \cite{HL82, Haskins-SLag-cones, HK07, Haskins-Kapouleas-gluing, Oh07}). Although examples of special Lagrangians with multiple conical singularities have been constructed from a restricted admissible family of base cones \cite[Example 6.12]{Pacini-gluing}, our approach provides a flexible setting for constructing examples of singular special Lagrangians starting from any finite family of special Lagrangian cones with the same Lagrangian angle. 

For example, we can apply Theorem \ref{thm: main-theorem} to any suitably arranged finite list of the Haskins special Lagrangian $T^2$-cones in $\mathbb{C}^3$ \cite{Haskins-SLag-cones, Ha04} to produce new singular examples. Perhaps more striking is that we can join copies of the graphical special Lagrangian cone in $\mathbb{C}^5$ discovered by Bhattacharya, Orriols, and Skorobogatova \cite[Theorem 1.4]{BOS26} to obtain infinite families of graphical special Lagrangians with isolated conical singularities (see Section 4.4). Each such graph is the gradient of a $C^{1,1}$ potential function $u : \Omega \subset \mathbb{R}^5 \rightarrow \mathbb{R}$ solving the phase zero special Lagrangian equation $\arctan(D^2u(x)) = 0$ in $\Omega$ in the viscosity sense, asymptotic to the cone potential at each singularity, which is not $C^2$. Thus, our theorem can be realized as a bridge principle in special Lagrangian geometry and for viscosity solutions to the fully nonlinear special Lagrangian equation. In the corollary below, we will denote the cone potential for the graph by $u_0: \mathbb{R}^5 \rightarrow \mathbb{R}$.
\begin{corollary}\label{cor: SLEsolution}
    Fix $N \in \mathbb{N}$. Then there is a bounded smooth domain $\Omega \subset \mathbb{R}^5$ and a viscosity solution $u: \Omega \rightarrow \mathbb{R}$ to the phase zero special Lagrangian equation with smooth boundary data such that $u \in C^{1,1}(\Omega) \setminus C^2(\Omega)$ and $u$ has precisely $N$ isolated singularities in $\Omega$. At each singularity, after translation and normalization, $u - u_0 = O(r^\nu)$ for some $\nu > 2$.
\end{corollary}
\begin{remark}
    When $m \leq 4$, all $C^{1,1}$ viscosity solutions are smooth by quadratic re-scaling and blow-up \cite{NV13}. Conversely, when $m \geq 5$ the Hausdorff dimension of the singular set is bounded above by $m - 5$, and when $m = 5$ interior singularities must be isolated \cite[Theorem 1.1]{BOS26} (see also \cite[Theorem 3.9]{Dim23}). Due to the boundary regularity theorem of Silvestre and Sirakov \cite[Theorem 1.3]{SS14}, this implies the singular set must be finite when $m = 5$. Thus, Corollary \ref{cor: SLEsolution} establishes the optimality of the regularity theory in the critical dimension: singularities may occur, and their number can be arbitrarily large.
\end{remark}
\begin{remark}
    Since the gradient graph of any $C^{1,1}$ viscosity solution to the special Lagrangian equation is an area-minimizing Lipschitz submanifold, we have also produced new singular minimizing solutions to the minimal surface system. To our knowledge, these are the first non-conical area-minimizing Lipschitz solutions to the minimal surface system with multiple prescribed isolated conical singularities.
\end{remark}

%\textcolor{blue}{I think that we might also consider making a corollary out of this. It think that it would give more relevance to this application.} - \textcolor{red}{I agree. I wanted your input first :p. I think it also demonstrates how powerful our theorem is, since a single new example leads to many, and I think it is quite pretty that both papers came out at the same time by coincidence lol. \\ I believe the framing above is set up well for subsequent work --- if we get the coassociative case we will have area-minimizing in the critical dimension. We need solutions on a ball to claim all size singular sets, since the system does not have such nice boundary regularity. It would also be interesting to produce examples with singularities accumulating at the boundary, as in one of your other papers. If we can do it, great, if not maybe we can improve the boundary regularity... just a thought!}\\
%\textcolor{blue}{maybe we can add a line to the abstract to describe this application.}

In \cite{CHS}, Caffarelli, Hardt, and Simon constructed the first known non-conical embedded minimal hypersurfaces with isolated singularities, by perturbing truncated regular minimal hypercones (see also \cite{Simon-singularities-extrema}). Shortly after, Hardt and Simon showed that for a \emph{strictly minimizing} hypercone \cite[Section 3]{HS85}, the minimizing condition can be preserved after perturbation \cite[Theorem 4.4]{HS85}. Building on \cite{CHS}, Smale produced in \cite{Sm89} the first connected minimal hypersurfaces with arbitrarily many isolated singularities (see \cite{Sm93} for the higher-codimensional case) by joining together truncated cones by ``$\varepsilon$-bridges" and showing that such configurations can be perturbed into a minimal surface. Even when the starting cones are area-minimizing, however, it is not clear whether Smale's construction yields an area-minimizing minimal submanifold \cite[Section 5]{Sm89}. Since most known area-minimizers are calibrated, it makes sense to study the problem in some calibrated settings. We therefore develop an analogue of Smale’s construction in the special Lagrangian setting after extending the Caffarelli--Hardt--Simon framework to special Lagrangian cones. As a result, the manifolds constructed in Theorem \ref{thm: main-theorem} are all area-minimizing (see Remark \ref{rmk: mean-curvature-lagrangian-angle}). One limitation of the chosen approach is that all of our examples necessarily have boundary. It would be interesting to produce examples as in Theorem \ref{thm: main-theorem} with more general boundary data, or complete/global solutions. \\

The idea of the proof of Theorem \ref{thm: main-theorem} is the following. Assume that the truncated cones $C^1,\ldots, C^N$ all have constant Lagrangian angle $\theta_0$.
First, for any $\varepsilon>0$ one constructs \emph{approximate solutions} $L^\varepsilon$, i.e. Lagrangian (but not special Lagrangian) connected submanifolds of $\C^m$ with boundary which contains the truncated cones $C^i$, and whose Lagrangian angle defect $\Theta_0^\varepsilon:=\theta_{L^\varepsilon}-\theta_0$ is small in $L^p$ (for any $p\geq 1$) and in $C^{0,\alpha}$ (see \eqref{bridgeest1} for the precise statement). Here, $\theta_{L^\varepsilon}$ represents the Lagrangian angle of $L^\varepsilon$. This can be obtained by joining the truncated cones with flat strips contained in the planes $T^i$, by a cut-and-paste argument.
Next, one notes that the Lagrangian perturbations of $L^\varepsilon$ can be identified, via a Weinstein neighborhood argument (Proposition \ref{prop:weinstein-approximate-solutions}), with sections in $T^\ast L^\varepsilon$. 
As we are interested in perturbations that remain asymptotic to the original cones, we will restrict our attention to one-forms with suitable weighted decay conditions near conical points. For a closed form $\xi$, let $L^\varepsilon_\xi$ be the perturbation of $L^\varepsilon$ induced by $\xi$, and let $\Theta(\xi)(x)$ denote the Lagrangian angle of $L^\varepsilon_\xi$ (at the point corresponding to $x$).
We will show that for any $u\in C^{2,\alpha}_{\text{loc}}(L^\varepsilon)$,
\begin{align*}
    \Theta(du)=\theta_{L^\varepsilon}+d^\ast du+R(u)=\theta_{L^\varepsilon}+\Delta_{L^\varepsilon}u+R(u).
\end{align*}
Therefore, any solution $u$ of
\begin{align}\label{eq: PDE-for-u-intro}
    \Delta_{L^\varepsilon}u=-\Theta_{0}^\varepsilon-R(u)
\end{align}
decaying sufficiently fast at the conical points induces a Lagrangian submanifold $L^\varepsilon_{du}$ satisfying the conclusion of Theorem \ref{thm: main-theorem}.
To find solutions of \eqref{eq: PDE-for-u-intro} we argue as follows: for $\psi\in C^{2,\alpha}(\partial L^\varepsilon)$ supported away from the bridges and for $u\in C^{2,\alpha}_{\text{loc}}(L^\varepsilon)$ in an appropriate space (encoding the decay properties around the singularities), let $\mathcal{P}(u)$ be the unique solution $v$ of
\begin{align*}
\begin{cases}
    \Delta_{L^\varepsilon} v=-\Theta_{0}^\varepsilon-R(u) & \text{in } L^\varepsilon,\\
    \Pi_\nu^\varepsilon v=\Pi_\nu^\varepsilon\psi & \text{on } \partial L^\varepsilon,
\end{cases}
\end{align*}
where $\Pi_\nu^\varepsilon$ is the projection onto the $L^2$-orthogonal complement of finitely many elements in $L^2(\partial L^\varepsilon)$ (see \eqref{eq: epsilon-projection} for the precise definition).
We will then show that for sufficiently small boundary data and for $\varepsilon$ sufficiently small, one can find a fixed point of the operator $\mathcal{P}$, i.e. a solution of \eqref{eq: PDE-for-u-intro}, by an application of Schauder's fixed point theorem. \\

The strategy outlined above is based on the works of Smale \cite{Sm89, Sm93}. There are, however, a few notable differences.
Smale also starts by constructing approximate solutions $M^\varepsilon$, i.e. unions of minimal cones joined by small strips (i.e. the $\varepsilon$-bridges), whose mean curvature $H_{M^\varepsilon}$ is small in $L^p$ (for $p\geq 1$). He then identifies perturbations of $M^\varepsilon$ with sections of the normal bundle of $M^\varepsilon$ in appropriate function spaces. Setting $H(u)$ for the mean curvature of the perturbation induced by a section $u$, he shows that
\begin{align*}
    H(u)=H_{M^\varepsilon}+L_{M^\varepsilon}(u)+Q(u),
\end{align*}
where $L_{M^\varepsilon}$ is the Jacobi operator of $M^\varepsilon$ and $Q$ is a non-linear remainder term, which satisfies good estimates in terms of $u$ and its derivatives \cite[see Proposition 3.3]{Sm93}. Then minimal perturbations of $M^\varepsilon$ correspond to sections $u$ such that
\begin{align}\label{eq: PDE-minimal}
    L_{M^\varepsilon}(u)=-H_{M^\varepsilon}-Q(u),
\end{align}
and following ideas from \cite{CHS}, Smale shows that when $\varepsilon>0$ is small enough, one can obtain solutions of \eqref{eq: PDE-minimal} combining the Fredholm theory for conical operators with  a fixed point argument.

A first distinction is that, while the linearized operator $L_{M^\varepsilon}$ (the Jacobi operator of $M^\varepsilon$) in Smale's approach is a second order operator on the normal sections, the linearization of the Lagrangian angle is a first order operator on the space of closed one-forms, which represent the allowed perturbations in our setting.
\begin{remark}
    We will see (Lemma \ref{lem: decaying-closed-forms-are-exact} and Remark \ref{rmk: decaying-closed-forms-are-exact}) that 
    closed forms on $L^\varepsilon$ vanishing at the tip are exact. Here, we will work with exact 1-forms for which the primitive function decays sufficiently fast at the cone tips. This will allow us to work with scalar functions (thought of as primitives of one-forms) and to apply directly the available Fredholm theory for conical operators of second order (see \cite{Mazzeo-Edge,LockhartMcOwen, Sm93}, for a modern exposition see \cite[Section 2]{FMM}).
\end{remark} 
A second distinction is that, in our setting, the control of the non-linear remainder $R(u)$ in terms of the second derivatives of $u$ is not linear as for $Q(u)$ in \cite[Proposition 3.3]{Sm93}, but superlinear; therefore we cannot apply directly the same strategy (in particular the estimates on page 39 therein). To circumvent this issue, we work with flat Lagrangian bridges inspired by those in \cite{Dim25}. This allows us to obtain smallness of the Lagrangian angle of the approximate solution not only in $L^p$, but also in $C^{0,\alpha}$. The $C^{0,\alpha}$-control allows us to apply Schauder estimates to control $R(u)$ in the desired norms. The downside of this approach is that our bridges are not as flexible as those in \cite{Sm93}, and in general do not allow us to connect two arbitrary cones with the same Lagrangian angle. \\

As a first step, we tried to understand the deformation theory of a single truncated Lagrangian cone. In this setting, we obtained a special Lagrangian analogue of the perturbation result for minimal cones of Caffarelli, Hardt, and Simon \cite{CHS}.
We include the result and its proof as we believe that it provides a good illustration of some of the main ideas of the argument of Theorem \ref{thm: main-theorem} in a simpler setting, and that the result itself may be of independent interest.
\begin{theorem}\label{thm: CHS-intro}
    Let $C$ be a truncated regular special Lagrangian cone in $\C^m$ centered at the origin, with link $\Sigma$. There exists a
finite-dimensional linear subspace
$W\subset C^{2,\alpha}(\Sigma)$
such that, for every $\psi\in C^{2,\alpha}(\Sigma)$ sufficiently small,
there exist a function
$u\in C^{2,\alpha}_{\mathrm{loc}}(C)$
and an element $w\in W$ satisfying
$u|_\Sigma=\psi+w$. Moreover, the Lagrangian submanifold $L_{du}$ induced by $du$ through a Weinstein neighborhood map is special Lagrangian and has an isolated conical singularity at the origin modeled on $C$. 
\end{theorem}

The existence and deformation theory of singular special Lagrangians have been extensively studied, especially in order to get a better understanding of the moduli space (in particular, its boundary) of a compact special Lagrangian submanifold in a Calabi-Yau ambient. In the smooth compact case, McLean \cite{McLean} showed that the moduli space is unobstructed and has tangent space canonically identified with the space of harmonic $1$-forms on $L$; in particular, its dimension is $b_1(L)$.

In a series of works \cite{joyce-1, joyce-2, Joyce-CS-III, Joyce-CS-IV,Joyce-CS-V}, Joyce studied the deformation theory of compact special Lagrangians with singularities modeled on cones. Under this framework, he studied the deformations of special Lagrangian cones. Theorem \ref{thm: CHS-intro} above is closely related to \cite[Section 6]{joyce-2}.
Joyce also showed that, under suitable assumptions, singularities can be ``smoothed out" by a gluing procedure. Pacini \cite{Pacini-moduli-spaces, Pacini-gluing, Pacini-Uniform} later developed the deformation theory for special Lagrangian conifolds (i.e. manifolds which can have singularities or ends asymptotic to special Lagrangian cones) and studied gluing operations, by which he constructed several examples of conifolds. In particular, he showed in \cite[Theorem 6.10, Example 6.12]{Pacini-gluing} that, under suitable assumptions, the union of stable\footnote{In this context, stability should be interpreted in the sense of Joyce (e.g. \cite[Definition 3.6]{Joyce-CS-V}) as opposed to the variational notion of stability appearing in the minimal surface literature.} special Lagrangian cones and planes satisfying Lawlor's angle condition can be glued together using Lawlor's necks into a special Lagrangian manifold with multiple singularities. However, stability has only been verified for a limited number of explicit examples (e.g.  \cite{Joyce-CS-V, Ha04, Oh07}). 

As observed before, non-flat Lagrangian cones exist in complex dimension $m\geq 3$. When the ambient is a K\"ahler manifold of complex dimension $2$, one could consider similar questions for Hamiltonian stationary surfaces. Such surfaces have isolated conical singularities, and the tangent cones are classified \cite{Schoen-Wolfson, Pigati-Riviere}. A result similar to Theorem \ref{thm: main-theorem} in this setting was obtained in \cite{GOR}. Recently \cite{Dim25}, Smale's method was used to produce families of four dimensional variationally stable (possibly non-minimizing) graphical examples from copies of the Lawson-Osserman cone\footnote{See \cite{Dim25} and the references therein for a discussion of the Lawson-Osserman cone.}, demonstrating that the bound on the Hausdorff dimension of the singular set of a Lipschitz stationary solution to the minimal surface system (e.g. \cite[Theorem 3.7]{Dim23}) cannot, in general, be improved. The most important step in the construction is a flattening procedure at the boundary of a truncated regular minimal cone, allowing one to solve the fixed point problem in low dimensions while preserving graphicality. Since we use a modification of this technique to construct our Lagrangian approximate solutions, it is of fundamental importance to the present paper as well.\\ 

Another noteworthy point is that the bridge principle does not require perturbation of the ambient metric, allowing us to produce singular examples in $\mathbb{C}^m$. This feature was also exploited in \cite{Dim25} to construct singular minimal graphs. In comparison, Simon constructed variationally stable minimal hypersurfaces with arbitrary closed singular sets having Hausdorff dimension $\leq m-7$ \cite{Sim23}. Shortly after, Liu constructed area-minimizing examples (calibrated, in fact) with high-codimension and singular set with any Hausdorff dimension $\leq m-2$, thus resolving a conjecture of Almgren \cite{Liu25}. Although \cite{Sim23, Liu25} allow for significantly more flexibility than Smale's bridge principle in prescribing the singular set, they both rely on perturbing the ambient metric to a smooth metric.\\ 

Finally, it is worth mentioning that Brian White developed an alternate approach to proving bridge principles using geometric measure theory that is quite broad \cite{Wh941, Wh942}, allowing one to bridge locally area-minimizing submanifolds without restrictions on the singular set --- provided they uniquely solve the
Plateau problem for their boundary data in some open set. However, one cannot guarantee singularity preservation. We refer the reader to \cite[Section 1 \& Section 6]{Dim25} for a thorough comparison of Smale's and White's approaches. For both early and modern accounts of the history of the bridge principle, see \cite{Wh941} and \cite{Dim25}, respectively.

\begin{ack*} 
We would like to thank Rafe Mazzeo and Rick Schoen for their interest in the problem and helpful discussions, and Arunima Bhattacharya and Tommaso Pacini for helpful comments. 
\end{ack*}

\begin{fund*}
    B.D. was partially supported by Connor Mooney's NSF CAREER Grant DMS-2143668, as well as NSF RTG Grant DMS-2342135. F. G. was supported by the Swiss National Science Foundation (SNSF)
through the Postdoc.Mobility grant, project number 230344.
\end{fund*}

\begin{ai*}
    The authors used ChatGPT 5.6 Sol as an auxiliary tool for literature searches, computational checks, and editorial assistance. This manuscript contains no AI generated text, and the authors take full responsibility for its contents.
\end{ai*}

\section{Preliminaries}
We prove some basic facts about Lagrangian and special Lagrangian submanifolds, with an emphasis on cones. We also define conically singular Lagrangian submanifolds, and define the weighted H\"older and Sobolev spaces we will work in.
\subsection{Lagrangian cones}
\begin{definition}
    Let $(M^{2m},\omega)$ be a \emph{symplectic manifold}, i.e. a $2m$-dimensional manifold $M^{2m}$ endowed with a closed, non-degenerate $2$-form $\omega$, called a \emph{symplectic form}. A submanifold $L^m$ of dimension $m$ is called \emph{Lagrangian} if $\omega\vert_L=0$.
\end{definition}
In this paper, we will work with two different symplectic manifolds: $M=\C^m$ with the symplectic form $\omega=\Sigma_{i=1}^m dx_i\wedge dy_i$ and $M=T^\ast L$, for a manifold $L^m$, with the canonical symplectic form $\omega_{\text{can}}=\sum_{i=1}^m dx_i\wedge dp_i$. Here, $x_1,...,x_m$ are local coordinates on $L^m$ and $x_1,...,x_m,p_1,...,p_m$ are the induced local coordinates on $T^\ast L^m$.\\
If $C\subset \mathbb C^m$ is a cone with vertex at the origin, we write
\begin{align*}
C^\ast:=C\setminus{0}, \qquad \Sigma:=C\cap \mathbb{S}^{2m-1}.
\end{align*}
The submanifold $\Sigma$ is called the \emph{link} for $C$. Throughout the paper, all cones are assumed to be regular, meaning that their links $\Sigma$ are smooth embedded submanifolds of $\mathbb{S}^{2m-1}$. Also, by an abuse of notation, we will sometimes write $C$ for $C^\ast$. We will sometimes write $C(\Sigma)$ to emphasize that $C^\ast$ is determined by $\Sigma$ via radial dilation.

For $r>0$, set 
\[
C_{r}:=C^*\cap\{0<|z|<r\}.
\]
Let $g_C$ be the cone metric induced from $\C^m$ on $C^\ast$, and for $\lambda>0$, let $\delta_\lambda:\mathbb C^m\to \mathbb C^m$ denote the dilation $\delta_\lambda(z)=\lambda z$. Then
\[
\delta_\lambda^\ast g_C=\lambda^2 g_C.
\]
We denote by $\nabla$ the Levi--Civita connection of $g_C$ on $C^\ast$. Define $\tilde\delta_\lambda:T^\ast C^\ast\to T^\ast C^\ast$ by
\[
\tilde\delta_\lambda(x,\xi):=\bigl(\lambda x,\ \lambda^2(d\delta_\lambda^{-1})^\ast\xi\bigr).
\]
Direct computation gives $\tilde\delta_\lambda^\ast\omega_{\mathrm{can}}=\lambda^2\omega_{\mathrm{can}}$.

Let $|\xi|_{g_C}$ denote the norm of $\xi\in T_x^\ast C^\ast$ induced by $g_C$. In order to describe the Lagrangian deformations of a truncated cone $C_1$, it will be useful to identify a neighborhood of $C_1$ with a region of its cotangent bundle, as described in the following result.
\begin{proposition}\label{prop:weinstein}
There exist $\varepsilon_0>0$, a conical open neighborhood (i.e. dilation invariant) 
\[
\mathcal{U}_{\varepsilon_0} := \{(x,\xi)\in T^\ast C_{1}:\ |\xi|_{g_C}<\varepsilon_0 \lvert x\rvert\},
\]
a conical open neighborhood $U\subset \C^m\setminus\{0\}$ of $C_{1}$, and a symplectomorphism
$\Psi:(\mathcal{U}_{\varepsilon_0},\omega_{\mathrm{can}})\to (U,\omega)$
such that:
\begin{enumerate}
\item $\Psi(x,0)=x$ for all $x\in C_{r}$;
\item $\Psi\circ \tilde\delta_\lambda = \delta_\lambda\circ \Psi$ for all $\lambda>0$;
\item $\Psi^\ast\omega=\omega_{\mathrm{can}}$ on $\mathcal{U}_\varepsilon$.
\end{enumerate}
\end{proposition}
This result corresponds to \cite[Theorem 4.3]{joyce-1}, and can be deduced from Proposition \ref{prop:weinstein-approximate-solutions}.

In particular, $L\subset \mathcal{U}_{\varepsilon_0}$ is a Lagrangian submanifold if and only if $\Psi(L)\subset U$ is Lagrangian.
Note that a section $s\in \mathcal{U}_{\varepsilon_0}\subset T^\ast C_1$ is Lagrangian if and only if it is closed as a $1$-form on $C_1$. Therefore the Lagrangian deformations of $C_1$ (staying sufficiently close to $C_1$) are in one-to-one correspondence with the closed $1$-forms in $\mathcal{U}_{\varepsilon_0}$.
The next lemma shows that all such $1$-forms are actually exact, since sections in $\mathcal{U}_{\varepsilon_0}$ decay sufficiently fast.
\begin{lemma}\label{lem: decaying-closed-forms-are-exact}
    Let $\xi\in \Omega^1(C_1)$ with $d\xi=0$ and
    \begin{align*}
        \lvert \xi_x\rvert\leq c\lvert x\rvert^{-1+\varepsilon},
    \end{align*}
    for some constants $\varepsilon, c>0$. Then $\xi$ is exact.
\end{lemma}
\begin{proof}
    Let $\gamma$ be any closed loop in $C_1$ (recall that $C_1$ does not include the tip). Note that $[\gamma]=[r\gamma]\in H_1(C_1,\mathbb{Z})$, for any $r\in (0,1)$. Therefore, since $d\xi=0$,
    \begin{align*}
        \left\lvert\int_\gamma \xi\right\rvert=\left\lvert\int_{r\gamma}\xi\right\rvert\leq CL(\gamma)r (\dist(\gamma, 0) r)^{-1+\varepsilon}.
    \end{align*}
    As the right hand side tends to zero as $r\to0$, we conclude that $\int_\gamma \xi=0$. Since this holds for any loop $\gamma$ in $C_1$, we conclude that $\xi$ is exact.
\end{proof}

\begin{remark}\label{rmk: decaying-closed-forms-are-exact}
    The result remains true for closed $1$-forms on $L^\varepsilon$, where $L^\varepsilon$ consists of a union of Lagrangian cones, joined by bridges, under the assumption that the bridges do not form any closed loop.
\end{remark}

\subsection{The Lagrangian angle}

Let $\Omega = dz_1\wedge\cdots\wedge dz_m$ denote the standard holomorphic volume form in $\C^m$.
\begin{definition}
    If $L\subset\C^m$ is a Lagrangian submanifold and $\iota$ denotes the inclusion of $L$ in $\C^m$, the \emph{Lagrangian angle} $\theta: L\to  \R/2\pi\mathbb{Z}$ of $L$ is determined by
    \begin{align*}
        \iota^\ast\Omega=e^{i\theta}\lvert\iota^\ast\Omega\rvert.
    \end{align*}
    A Lagrangian submanifold is called \emph{special Lagrangian} if its Lagrangian angle $\theta$ is constant.
\end{definition}
\begin{remark}\label{rmk: mean-curvature-lagrangian-angle}
    For a Lagrangian submanifold $L\subset \C^m$ with Lagrangian angle $\theta$, the mean curvature vector $H$ of $L$ satisfies
    \begin{align*}
        H=J\nabla \theta,
    \end{align*}
    where $J$ denotes the complex multiplication in $J$; see for instance Lemma 2.1 in \cite{Thomas-Yau}. In particular, a special Lagrangian submanifold $L\subset\C^m$ is minimal. In fact, if such a manifold has Lagrangian angle $\theta_0$, then it is calibrated by $\Re(e^{-i\theta_0}\Omega)$, and therefore it is area-minimizing.
\end{remark}
Let $C$ be a regular Lagrangian cone (identified with $C^*$) with Lagrangian angle $\theta_0$ and choose $\varepsilon_0$ as in Proposition \ref{prop:weinstein}. For a $1$-form $\xi \in T^\ast C_1$ such that $|\xi|_{g_C}<\varepsilon_0 r$, let
\[
\Gamma_\xi:=\{(x,\xi_x):x\in C_{1}\}\subset \mathcal{U}_{\varepsilon_0},
\quad
L_\xi:=\Psi(\Gamma_\xi)\subset \C^m,
\]
and define
\begin{align*}
    \iota_\xi:C_{1}\to L_\xi,\quad \iota_\xi(x):=\Psi(x,\xi_x).
\end{align*} 
For any $f\in C^{2,\alpha}(C_{1})$, define the \emph{pulled-back Lagrangian angle} $\theta(f):C_{1}\to\R/2\pi\mathbb{Z}$ by
\[
\iota_{df}^\ast \Omega = e^{i\theta(f)}\,\big|\iota_{df}^\ast\Omega\big|.
\]
When $df=0$, this recovers $\theta(0)=\theta_0$. 

Consider the variation
\begin{align*}
    F_t: C_{1}\to\mathbb{C}^m,\quad x\mapsto \iota_{tdf}(x).
\end{align*}
Note that for any $x\in C_{1}$,
\begin{align}\label{eq: hamiltonian-variation}
    \frac{d}{dt}\bigg\vert_{t=0}F_t(x)=\frac{d}{dt}\bigg\vert_{t=0}\Psi(x, tdf(x))=D\Psi(x,0)[0,df(x)]=J\nabla f(x).
\end{align}
For the last step, we used the fact that by Properties 1. and 3. in Proposition \ref{prop:weinstein}, for any vector field $X$ tangent to $C_1$ we have
\begin{align*}
    D\Psi(x,0)[0, df(x)]\cdot JX=-\omega(D\Psi(x,0)[0,df(x)], X)=-\omega_{\text{can}}(df(x), X)=\nabla f(x)\cdot X
\end{align*}
and similarly
\begin{align*}
    D\Psi(x,0)[0, df(x)]\cdot D\Psi(x,0)[0, dg(x)]=0
\end{align*}
for any $g\in C^1(C_1)$.

We can obtain an explicit expression for the pulled-back Lagrangian angle as follows.
Let
\(e_1,\ldots,e_m\) be a local oriented \(g_C\)-orthonormal frame on \(C\). If
\(f\) is sufficiently small in \(C^2\), then
\[
    d\iota_{df}(e_i)
    =
    D_x\Psi(x,df_x)[e_i]
    +
    D_v\Psi(x,df_x)[\nabla_{e_i}df],
\]
where \(D_x\Psi\) and \(D_v\Psi\) denote the horizontal and vertical
differentials of \(\Psi\). Therefore
\[
    \iota_{df}^*\Omega(e_1,\ldots,e_m)
    =
    {\det}_{\C}\Big[
    D_x\Psi(x,df_x)[e_i]
    +
    D_v\Psi(x,df_x)[\nabla_{e_i}df]
    \Big]_{i=1}^m .
\]
Thus, after fixing the branch of the argument near \(\theta_0\), 
\[
    \theta(f)(x)
    =
    \operatorname{arg}
    {\det}_{\C}\Big[
    D_x\Psi(x,df_x)[e_i]
    +
    D_v\Psi(x,df_x)[\nabla_{e_i}df]
    \Big]_{i=1}^m.
\]

The linearization of $\theta(f)$ is computed in the following proposition.
\begin{proposition}\label{prop: linearization} 
    Let \(f\in C^{2,\alpha}_{\mathrm{loc}}(C^\ast)\). Then, after choosing the
branch of the Lagrangian angle near \(\theta_0\), we have
\[
    \left.\frac{d}{ds}\right|_{s=0}\theta(sf)
    =
    \Delta_{g_C}f .
\]
\end{proposition}
\begin{proof}
Fix \(x\in C^\ast\), and choose a local oriented \(g_C\)-orthonormal frame
\(e_1,\ldots,e_m\) with \(\nabla^C e_i(x)=0\). Set
\[
    F_s:=\iota_{sdf},
    \qquad
    V_i(s):=dF_s(e_i),
    \qquad
    Z(s):={\det}_\C(V_1(s),\ldots,V_m(s)).
\]
Then $\theta(sf)(x)=\operatorname{arg} Z(s)$,
where the branch of \(\operatorname{arg}\) is chosen near \(\theta_0\).\\
We have
\[
    \left.\frac{d}{ds}\right|_{s=0}\theta(sf)(x)
    =
    \operatorname{Im}\frac{Z'(0)}{Z(0)}=\operatorname{Im} Z'(0).
\]

By \eqref{eq: hamiltonian-variation}, there holds
\[
    \left.\frac{d}{ds}\right|_{s=0}F_s
    =
    J\nabla f.
\]
Hence, at \(x\) we have
\[
    \dot V_i(0)
    =
    J\nabla_{e_i}\nabla f .
\]
Decomposing the ambient derivative into tangent and normal components, we obtain
\[
    \dot V_i(0)
    =
    J\nabla^C_{e_i}\nabla f
    +
    JA(e_i,\nabla f),
\]
where $A$ is the second fundamental form for $C$. Write
\[
    \nabla^C_{e_i}\nabla f
    =
    \sum_{j=1}^m f_{ij}e_j,
    \qquad
    f_{ij}:=(\nabla df)(e_i,e_j),
\]
and
\[
    JA(e_i,\nabla f)
    =
    \sum_{j=1}^m a_{ij}e_j .
\]
Here, we are using that \(C^\ast\) is Lagrangian so that
\(JA(e_i,\nabla f)\in T_xC^\ast\). Therefore,
\[
    V_i(s)
    =
    e_i
    +
    s\sum_{j=1}^m a_{ij}e_j
    +
    s\sum_{j=1}^m f_{ij}Je_j
    +
    o(s).
\]
Then
\[
    Z(s)
    =
    Z(0)\det\big(I+s(a_{ij}+i f_{ij})+o(s)\big).
\]
Consequently,
\[
    Z'(0)
    =
    \operatorname{tr}(a_{ij}+i f_{ij})
    =
    \sum_{i=1}^m a_{ii}
    +
    i\sum_{i=1}^m f_{ii}.
\]
Taking imaginary parts, we get
\[
    \left.\frac{d}{ds}\right|_{s=0}\theta(sf)(x)
    =
    \sum_{i=1}^m f_{ii}
    =
    \operatorname{tr}_{g_C}\nabla df
    =
    \Delta_{g_C}f(x).
\]
\end{proof}

Suppose $C^\ast$ has constant Lagrangian angle $\theta_0\in [0,2\pi)$ and, for $f\in C^{2,\alpha}_{\text{loc}}(C^\ast)$, set
\begin{align}\label{eq: def-rest-term-angle}
    R(f):=\theta(f)-\theta_0-\Delta_{g_C} f.
\end{align}
By Proposition \ref{prop: linearization}, the linearization of the
Lagrangian angle at $\theta_0$ is \(\Delta_{g_C}f\). Hence,
\begin{align}\label{eq: properties-of-R}
    R(0)=0,
    \qquad
    \left.\frac{d}{ds}\right|_{s=0}R(sf)=0 .
\end{align}
\begin{lemma}\label{lem:R-rescaling}
For any
$f\in C^{2,\alpha}(C_{r})$, define
\begin{align}
\label{eq: def-f-lambda}
f_\lambda := \lambda^{-2}\,f\circ \delta_\lambda.
\end{align}
Then for every \(\lambda>0\) we have
\[
    R(f)\circ\delta_\lambda=R(f_\lambda) \text{ on } C_{\lambda^{-1}r}.
\]
\end{lemma}
\begin{proof}
By Proposition \ref{prop:weinstein}, $   \Psi\circ\widetilde \delta_\lambda=\delta_\lambda \circ\Psi$ .
Moreover,
\[
    df_\lambda(x)
    =
    \lambda^{-2}(d\delta_\lambda)^*df(\delta_\lambda x),
\]
or equivalently
\[
    df(\delta_\lambda x)
    =
    \lambda^2(d\delta_\lambda^{-1})^*df_\lambda(x).
\]
Therefore
\[
    \widetilde\delta_\lambda(x,df_\lambda(x))
    =
    (\delta_\lambda x,df(\delta_\lambda x)),
\]
and hence
\[
    \iota_{df}(\delta_\lambda x)
    =
    \delta_\lambda\big(\iota_{df_\lambda}(x)\big).
\]
Note that dilations do not affect the Lagrangian angle. Therefore $\theta(f)(\delta_\lambda x)=\theta(f_\lambda)(x)$.
Also, since \(\delta_\lambda^*g_C=\lambda^2g_C\), one has
\[
    \Delta_{g_C}(\lambda^{-2}u\circ\delta_\lambda)
    =
    (\Delta_{g_C}u)\circ\delta_\lambda .
\]
Applying this to \(u=f\), we get
\[
    \Delta_{g_C}f_\lambda
    =
    (\Delta_{g_C}f)\circ\delta_\lambda .
\]
Combining the previous identities gives
\begin{align*}
     R(f)(\delta_\lambda x)
    =
    R(f_\lambda)(x).
\end{align*}
\end{proof}
We also record the local expression of \(R\). As above, let \(e_1,\ldots,e_m\) be a local oriented \(g_C\)-orthonormal frame on \(C^\ast\).
After shrinking the Weinstein neighborhood if necessary, there
is a smooth function
\begin{align}\label{eq: def-of-H}
    G:
    T^*C_r\oplus \operatorname{Sym}^2T^*C_r
    \supset \mathcal V
    \longrightarrow \R
\end{align}
defined in a neighborhood $\mathcal V$ of the zero section such that
\[
    R(f)(x)
    =
    G\big(x,df_x,\nabla df_x\big).
\]
Here,
\[
\begin{aligned}
    G(x,v,w)
    :=
    \operatorname{arg}
    \det_{\C}\Big[
    D_x\Psi(x,v)[e_i]
    +
    D_v\Psi(x,v)[w(e_i,\cdot)]
    \Big]_{i=1}^m-\theta_0-\operatorname{tr}_{g_C(x)} w
\end{aligned}
\]
for \(v\in T_x^*C_r\), \(w\in \operatorname{Sym}^2T_x^*C_r\), and
$\operatorname{tr}_{g_C(x)}w=\sum_{i=1}^m w(e_i,e_i)$.
The observation after Proposition \ref{prop: linearization} implies that
\begin{align}
\label{eq: vanishing-H}
    G(x,0,0)=0,
    \qquad
    D_{(v,w)}G(x,0,0)=0.
\end{align}

\subsection{The weighted spaces}
We will work with various weighted H\"older and weighted Sobolev spaces on Lagrangian cones and conically singular submanifolds. In accordance with \cite{Sm93}, we will identify a given regular Lagrangian cone $C^\ast$ having link $\Sigma$ with a warped product cone: 
\[C^\ast \simeq (0,\infty) \times_r \Sigma.\]
Thus, any point $x = ry \in C^\ast$ can be identified with a pair $(r,y) \in (0,\infty) \times \Sigma$. When we would like to emphasize that $C^\ast$ is determined by $\Sigma$, we will write $C(\Sigma)$ for $C^\ast$.

With the identification above, the cone metric can be written
\[g_C = dr^2 + r^2 g_\Sigma,\]
where $g_\Sigma$ denotes the induced metric on the link, and we can identify the truncated cones $C_{\sigma}(\Sigma)$ for $\sigma > 0$ with
\[C_{\sigma}(\Sigma) \simeq (0,\sigma] \times_r \Sigma.\]
Annuli in $C^\ast$ are represented by 
\[[\sigma_1, \sigma_2] := [\sigma_1, \sigma_2] \times_r \Sigma \text{ for } 0 < \sigma_1 < \sigma_2 < \infty.\]
We can similarly define open and half-open annuli in $C^\ast$.

\subsubsection{Cones}
The weighted spaces will be defined only on the truncated cones $C_1$, although the definitions extend naturally to $C_r$ for any $r > 0$. Let $C_1 := C_1(\Sigma) = (0,1]\times_r \Sigma$ represent a warped product regular Lagrangian cone. Suppose $k \in \{0,1,2\}$, $\alpha \in (0,1)$, $\nu \in \R$, and $\sigma \in (0,\frac{1}{2}]$. Let $\|\cdot\|_{C^k([\sigma_1,\sigma_2])}$ and $\|\cdot\|_{C^{k,\alpha}([\sigma_1,\sigma_2])}$ denote the standard $C^k$ and $C^{k,\alpha}$ norms on functions defined on the annulus $[\sigma_1,\sigma_2] \times_r \Sigma \subset C_1$, and let $[\cdot]_{(\alpha);[\sigma_1,\sigma_2]}$ denote the usual H\"older semi-norm on functions on the same annulus. 

For $u \in C_{\text{loc}}^{k}(C_1)$, we define
\begin{align*}
    [u]_{k;\nu;\sigma} &:= \sum_{j + l \leq k}\|(r\partial_r)^j (\nabla^\Sigma)^l u\|_{C^0([\sigma, 2\sigma])} \, \sigma^{-\nu}\\
    \tnorm{u}_{k;\nu} &:= \sup_{\sigma \in (0,\frac{1}{2}]} [u]_{k;\nu;\sigma},
\end{align*}
and for $u \in C_{\text{loc}}^{k,\alpha}(C_1)$ we define
\begin{align*}
    [u]_{k,\alpha;\nu;\sigma} &:= \sum_{j + l \leq k}\|(r\partial_r)^j (\nabla^\Sigma)^l u\|_{C^0([\sigma, 2\sigma])} \, \sigma^{-\nu} + \sum_{j + l = k}\big[(r\partial_r)^j(\nabla^\Sigma)^lu\big]_{(\alpha); [\sigma,2\sigma]} \, \sigma^{-\nu + \alpha}\\
    \tnorm{u}_{k,\alpha;\nu} &:= \sup_{\sigma \in (0,\frac{1}{2}]} [u]_{k,\alpha;\nu;\sigma}.
\end{align*}
The weighted $C^k$ and H\"older spaces are then given by
\begin{align*}
    C^{k,\nu}(C_1) &:= \{u \in C_{\text{loc}}^k(C_1) : \tnorm{u}_{k;\nu} < \infty\} \\
    C^{k,\alpha, \nu}(C_1) &:= \{u \in C_{\text{loc}}^{k,\alpha}(C_1) : \tnorm{u}_{k,\alpha;\nu} < \infty\}.
\end{align*}
They are Banach spaces with their respective norms.
The same notation will be used for tensor fields on \(C_1\), with pointwise
norms computed using \(g_C\) and derivatives taken with respect to the
Levi-Civita connection of \(g_C\). In particular, expressions such as
\(\tnorm{du}_{1,\alpha;\nu}\) are understood in this sense.

The spaces $C_0^{k,\nu}(C_1)$ and $C_0^{k,\alpha,\nu}(C_1)$ will represent those functions in $C^{k,\nu}(C_1)$ and $C^{k,\alpha,\nu}(C_1)$, respectively, vanishing on \[\partial C_1 = \{1\} \times \Sigma \simeq \Sigma.\] 
We have the following relationships between the weighted H\"older spaces. 
\begin{lemma}\label{lem:derivative-weight}
For any $u\in C^{2,\alpha}_{\text{loc}}(C_{1})$ and any $\nu\in\R$, there is a constant $c:= c(\Sigma,\alpha)$ such that
\begin{align}\label{eq: est-fist-second-der}  
\tnorm{du}_{1,\alpha;\nu}\le c\tnorm{u}_{2,\alpha;\nu+1}.
\end{align}
Moreover, if $\nu<\nu^\prime$ we have
\begin{align}\label{eq: comparison-exponents}
\tnorm{du}_{1,\alpha;\nu}\le \tnorm{du}_{1,\alpha;\nu^\prime}.
\end{align}
\end{lemma}

We will also work in various weighted $L^2$ and Sobolev spaces. In our chosen coordinates on $C_1$
\[d\vol_{C} := r^{m-1} dr dy,\]
where $dy$ is the volume form on $\Sigma$. Let $L^2(C_1)$ be the usual $L^2$ space on $C_1$ with norm $\|\cdot\|_{L^2}$, and let $L_{\text{loc}}^2(C_1)$ be the measurable functions that are square integrable on compact subsets of $C_1$. We can similarly define the Banach spaces $L^p(C_1)$, and $L_{\text{loc}}^p(C_1)$ for any $p \geq 1$. The norm on $L^p(C_1)$ will be denoted by $\|\cdot\|_{L^p}$. For $\delta \in \R$, we then define the weighted $L^2$-spaces by
\[r^\delta L^2(C_1) := \{ u  \in L_{\text{loc}}^2(C_1) : r^{-\delta} u \in L^2(C_1)\}.\]
The norm on $r^\delta L^2(C_1)$ is 
\[\|u\|_{0; \delta} := \|r^{-\delta}u\|_{L^2(C_1)}.\]
The space $r^\delta L^2(C_1)$ is a Hilbert space with this norm and its natural inner product. 

For $k = 1,2$, $H^k(C_1)$ will represent the scale-invariant Sobolev spaces on $C_1$:  
\[H^k(C_1) := \Big\{u \in L^2(C_1) : \text{ weak derivatives order} \leq k \text{ exist, } \sum_{j + l \leq k} \|(r\partial_r)^j(\nabla^\Sigma)^lu\|_{L^2(C_1)}^2 < \infty\Big\}.\]
The spaces $H^k(C_1)$ are Hilbert spaces when endowed with the norm
\[\|u\|_{H^k(C_1)}^2 := \sum_{j + l \leq k} \|(r\partial_r)^j(\nabla^\Sigma)^lu\|_{L^2(C_1)}^2\]
and its associated inner product. Of course, we can define $H_{\text{loc}}^k(C_1)$ in the same manner as before.

For $k = 1,2$, the weighted Sobolev spaces $r^\delta H^k(C_1)$ can now be defined similarly to $r^\delta L^2(C_1)$:
\[r^\delta H^k(C_1) := \{u \in H_{\text{loc}}^k(C_1) : r^{-\delta}(r \partial_r)^j(\nabla^\Sigma)^lu \in L^2(C_1) \text{ whenever } j + l \leq k\},\]
with norm
\[\|u\|_{k; \delta}^2 := \sum_{j + l \leq k} \|(r \partial_r)^j(\nabla^\Sigma)^lu\|_{0;\delta}^2.\]
The spaces $r^\delta H^k(C_1)$ are Hilbert spaces with this norm and its natural inner product. We define $H_0^k(C_1)$ and $r^\delta H_0^k(C_1)$ to be the closure under the $H^k(C_1)$ and $r^\delta H^k(C_1)$ norms, respectively, of the space of smooth functions vanishing on the link $\Sigma$ of $C_1$.

\subsubsection{Conically singular submanifolds}
Suppose now $L \subset \mathbb{C}^m$ is a compact Lagrangian submanifold with non-empty boundary having finitely many isolated singularities $p_1,\ldots, p_N$. We will denote the singular set of $L$ by
\[\sing(L) := \{p_1,\ldots, p_N\} \subset \mathbb{C}^m.\]
We say that $L$ is \emph{conically singular} if at each $p_i$ there is a regular Lagrangian cone $C_1(\Sigma^i)$ such that $L$ is \emph{modeled on $C_1(\Sigma^i)$ at $p_i$}. In other words, near each point $p_i$ the submanifold $L$ can be identified with the graph of $du_i$ over $C_1(\Sigma^i)$ for some $u_i \in C^\infty(C_1(\Sigma^i))$ via a Weinstein neighborhood map, where $\nabla^k du_i \sim O(\rho^{\nu - k})$ for some $\nu > 1$, $k = 0,1$. The singular points $p_i$ are sometimes called the \emph{conic points} for $L$. Similarly to a cone, we identify a conically singular submanifold $L$ with $L \setminus \sing(L)$. 

In this paper, our primary conically singular submanifolds will be Lagrangian cones and \emph{approximate solutions} constructed as (unit length) Lagrangian cones $C_1(\Sigma^1),\ldots, C_1(\Sigma^N)$ joined together along their boundaries by thin Lagrangian bridges of width $O(\varepsilon)$. Coordinates on $L^\varepsilon$ will coincide with the cone coordinates on $L_i^\varepsilon := (0,1) \times \Sigma^i$ for each $i$, and for each $i$ $\rho_i$ will denote the geodesic distance in $L^\varepsilon$ from $p_i$. Let $\rho :L^\varepsilon \rightarrow (0,\infty)$ be any bounded smooth function which is equal to $\rho_i$ on $C_1(\Sigma^i)$ (hence, equals $1$ on $\Sigma^i$) for each $i = 1,\ldots, N$, and is constant on the $\varepsilon$-bridges outside of a small $\eta$-neighborhood of the patching region where the bridges meet the cone boundaries. Note that $\eta$ can be chosen so that it does not depend on $\varepsilon$. For each $\sigma \in (0,1)$, we define
\[L_\sigma^\varepsilon := \{ x \in L^\varepsilon : \rho(x) \geq \sigma\}.\]
That is, $L_\sigma^\varepsilon$ is formed from the cones truncated by a length $\sigma \in (0,1)$ away from the vertex. 

The weighted spaces on $L^\varepsilon$ are defined similarly to those on the cones. Set $A_0 := L_{4^{-1}}^\varepsilon$ and, for $0 < \sigma_1 < \sigma_2 < 1$, let 
\[[\sigma_1,\sigma_2] \times_{\rho_i} \Sigma^i \text { for each } i = 1,\ldots, N\]
be an annulus on $C_1(\Sigma^i)$ for each $i$. The ``annuli" in $L^\varepsilon$ are defined to be the union
\[[\sigma_1,\sigma_2] := \bigcup_{i = 1}^N [\sigma_1,\sigma_2] \times_{\rho_i} \Sigma^i.\]

For $u \in C_{\text{loc}}^k(L^\varepsilon)$ ($k = 0,1,2$) and $0 < \sigma < 1/4$,  we will define
\begin{align*}
    [u]_{k;\nu;\sigma}^i &:= \sum_{j + l \leq k}\|(\rho\partial_{\rho})^j (\nabla^{\Sigma^i})^l u\lvert_{C_1(\Sigma^i)}\|_{C^0([\sigma, 2\sigma])} \, \sigma^{-\nu}\\
    \tnorm{u}_{k;\nu} &:= \|u\|_{C^k(A_0)} + \sup_{\sigma \in (0,\frac{1}{4})} \max_{i = 1,\ldots,N} [u]_{k;\nu;\sigma}^i.
\end{align*}
Similarly, for $u \in C_{\text{loc}}^{k,\alpha}(L^\varepsilon)$ we may set
\begin{align*}
    [u]_{k,\alpha;\nu;\sigma}^i &:= \sum_{j + l \leq k}\|(\rho\partial_\rho)^j (\nabla^{\Sigma^i})^l u \lvert_{C_1(\Sigma^i)}\|_{C^0([\sigma, 2\sigma])} \, \sigma^{-\nu} + \sum_{j + l = k}\big[(\rho\partial_\rho)^j(\nabla^{\Sigma^i})^lu\lvert_{C_1(\Sigma^i)}\big]_{(\alpha); [\sigma,2\sigma]} \, \sigma^{-\nu + \alpha}\\
    \tnorm{u}_{k,\alpha;\nu} &:= \|u\|_{C^{k,\alpha}(A_0)} + \sup_{\sigma \in (0,\frac{1}{4})} \max_{i = 1,\ldots,N} \,[u]_{k,\alpha;\nu;\sigma}^i.
\end{align*}
The Banach spaces $C^{k,\nu}(L^\varepsilon)$ and $C^{k,\alpha,\nu}(L^\varepsilon)$ are then given by
\begin{align*}
    C^{k,\nu}(L^\varepsilon) &:= \{u \in C_{\text{loc}}^k(L^\varepsilon) : \tnorm{u}_{k;\nu} < \infty\} \\
    C^{k,\alpha, \nu}(L^\varepsilon) &:= \{u \in C_{\text{loc}}^{k,\alpha}(L^\varepsilon) : \tnorm{u}_{k,\alpha;\nu} < \infty\}.
\end{align*}
The spaces $C_0^{k,\nu}(L^\varepsilon)$ and $C_0^{k,\alpha,\nu}(L^\varepsilon)$ represent those functions in their respective spaces that vanish on $\partial L^\varepsilon$.
\begin{remark}
    Lemma \ref{lem:derivative-weight} holds for both the $C^{k,\nu}$ and $C^{k,\alpha,\nu}$ norms on $L^\varepsilon$.
\end{remark}

Define two norms:
\begin{align*}
    \|u\|_{0;\delta}^2 &:= \|\rho^{-\delta}u \|_{L^2(L^\varepsilon \setminus A_0)}^2 + \|u\|_{L^2(A_0)}^2 \text{ for } u \in L_{\text{loc}}^2(L^\varepsilon) \\
    \|u\|_{k;\delta}^2 &:= \sum_{i = 1}^N\sum_{j + l \leq k}\|\chi_{i} \cdot (\rho \partial_\rho)^{j}(\nabla^{\Sigma^i})^l u\|_{0;\delta}^2 + \|u\|_{H^k(A_0)}^2 \text{ for } u \in H_{\text{loc}}^k(L^\varepsilon),
\end{align*}
where in the second line $\chi_{i} : L^\varepsilon \rightarrow \R$ is the characteristic function of $(0,1/4] \times_{\rho_i} \Sigma^i \subset C_1(\Sigma^i)$ for each $i$. Elements of $\rho^\delta L^2(L^\varepsilon)$ and $\rho^\delta H^k(L^\varepsilon)$ are those functions in $L_{\text{loc}}^2(L^\varepsilon)$ and $H_{\text{loc}}^k(L^\varepsilon)$ for which their respective norms are finite. Both $\rho^\delta L^2(L^\varepsilon)$ and $\rho^\delta H^k(L^\varepsilon)$ are Hilbert spaces with the norms above and their associated inner products. The space $\rho^\delta H_0^k(L^\varepsilon)$ is the closure under $\|\cdot\|_{k;\delta}$ of the smooth functions vanishing on $\partial L^\varepsilon$, and $H_0^k(L^\varepsilon)$ is defined the same as in the conical case. 
\begin{remark}
    The norms defined in this section are equivalent to those found in \cite{Sm89, Dim25}. Moreover, with only slight modifications the stated definitions can be extended to general conically singular Lagrangian submanifolds $L$ (e.g. see \cite{Sm93}), and general vector bundles over $L$. We can also define the weighted spaces so that we may prescribe distinct weights at each conic point. However, in this case the weights $\nu$ and $\delta$ are viewed as $N$-vectors (e.g. see \cite{Sm89, Sm93}). Our main result (Proposition \ref{prop: fixed-point-theorem}) holds in this setting as well.
\end{remark}

\section{Perturbations of a special Lagrangian cone}
In this section, we will prove Theorem \ref{thm: CHS-intro}. For convenience, we restate the theorem below.
\begin{theorem}\label{thm: SL-CHS}
    Let $C_1$ be a truncated special Lagrangian cone in $\C^m$ centered at the origin, with link $\Sigma$.
There exists a
finite-dimensional linear subspace
$W\subset C^{2,\alpha}(\Sigma)$
such that, for every $\psi\in C^{2,\alpha}(\Sigma)$ sufficiently small,
there exist a function
$u\in C^{2,\alpha}_{\mathrm{loc}}(C)$
and an element $w\in W$ satisfying
$u|_\Sigma=\psi+w$. Moreover, the Lagrangian submanifold $L_{du}$ induced by $du$ through a Weinstein neighborhood map is special Lagrangian and has a singularity at the origin modeled on the cone $C_1$.
\end{theorem}
This is a special Lagrangian analogue of the main result of
\cite{CHS}. \\

In order to prove Theorem \ref{thm: SL-CHS} we will argue as follows. Assume that $C$ has constant Lagrangian angle equal to $\theta_0$. After a unitary rotation of the cone, we may assume that $\theta_0=0$.
Let $\nu\notin \Lambda_\Sigma$ with $\nu> 2$, where $\Lambda_\Sigma$ denotes the set of indicial roots of $\Delta_\Sigma$ (see Appendix \ref{ssec: existence-uniqueness} for the definition).
Let $0=\mu_0<\mu_1\leq \mu_2\leq...$ denote the eigenvalues of $\Delta_\Sigma$, and for any $j\in \mathbb{N}$ denote by $\phi_j$ the corresponding normalized eigenfunction.
We will denote by $\Pi_\nu$ the $L^2$- projection onto the orthogonal complement to the finite dimensional space generated by eigenfunctions $\phi_j$ for which the associated indicial root $\gamma_j$ satisfies $\gamma_j<\nu$ (see Subsection \ref{ssec: existence-uniqueness} for more details). The space $V$ in the theorem will correspond to the image of $\Pi_\nu$ in $C^{2,\alpha}(\Sigma)$. For the remainder of the section, we set $\Delta := \Delta_{g_C}$.

By Theorem \ref{thm: CHS-linear}, for any $f\in C^{0,\alpha, \nu-2}$ (for $\alpha \in (0,1)$) and for any $\psi\in C^{2,\alpha}(\Sigma)$, there exists a unique map $u\in C^{2,\alpha,\nu}(C_1)$ solving
\begin{align}\label{eq:dirichlet-u-section-4}
\begin{cases}
    \Delta u=f & \text{in } C_{1},\\
    \Pi_\nu u=\Pi_\nu\psi & \text{on } \Sigma.
\end{cases}
\end{align}
Moreover, $u$ satisfies the estimate
\begin{align}\label{eq: est-u-CHS}
    \tnorm{u}_{2,\alpha;\nu}
    \leq
    c\left(
       \tnorm{f}_{0,\alpha;\nu-2}
        +
        \|\Pi_\nu\psi\|_{C^{2,\alpha}(\Sigma)}
    \right).
\end{align}
For any \(u\) with
$\tnorm{du}_{1,\alpha;1}$
sufficiently small and $\tnorm{du}_{1,\alpha;\nu/2}<\infty$, let $\mathcal{P}_\psi(u)$ denote the unique solution in $C^{2,\alpha, \nu}(C_1)$ of
\begin{align*}\label{eq:dirichlet-u-section-4-for-v}
\begin{cases}
    \Delta v=-R(u) & \text{in } C_{1},\\
    \Pi_\nu v=\Pi_\nu\psi & \text{on } \Sigma.
\end{cases}
\end{align*}
Note that, if \(u\) lies in the domain of \(\mathcal P_\psi\) and $\mathcal{P}_\psi(u)=u$, then $L_{du}$ has constant Lagrangian angle equal to zero by \eqref{eq: def-rest-term-angle}, and therefore is a special Lagrangian submanifold as in Theorem \ref{thm: SL-CHS}. 

In order to prove the theorem, we will show that for any $\psi\in C^{2,\alpha}(\Sigma)$ sufficiently small, $\mathcal{P}_\psi$ has a fixed point.
To this end, we first prove two estimates for the nonlinear remainder $R(u)$ in Section \ref{ssec: est-Rf}.
The existence of a fixed point for $\mathcal{P}_\psi$ will then be proved in Section \ref{ssec: fixed-point-argument}.

\subsection{Estimates for the nonlinear remainder}\label{ssec: est-Rf}

We first prove the required estimates on a fixed annulus $A := (1/2,1) \times_r \Sigma$.

\begin{lemma}\label{lem:fixed-annulus-R}
There are constants \(c,\varepsilon_0>0\) such that, if
\[
    \|df\|_{C^{1,\alpha}(A)},\ \|dh\|_{C^{1,\alpha}(A)}
    \le \varepsilon_0,
\]
then
\begin{align}
    \|R(f)\|_{C^{0,\alpha}(A)}
    &\le
    c\|df\|_{C^{1,\alpha}(A)}^2, \label{eq:fixed-annulus-R-quadratic}\\
    \|R(f)-R(h)\|_{C^{0,\alpha}(A)}
    &\le
    c\bigl(
        \|df\|_{C^{1,\alpha}(A)}
        +
        \|dh\|_{C^{1,\alpha}(A)}
    \bigr)
    \|df-dh\|_{C^{1,\alpha}(A)}. \label{eq:fixed-annulus-R-lipschitz}
\end{align}
\end{lemma}

\begin{proof}
Recall that
\[
    R(f)(x)=G(x,df_x,\nabla df_x),
\]
where \(G\) is smooth near the zero section and satisfies
\[
    G(x,0,0)=0,
    \qquad
    D_{(v,w)}G(x,0,0)=0.
\]
Set
\[
    z_f:=(df,\nabla df),
    \qquad
    z_h:=(dh,\nabla dh).
\]
Note that
\[
    \|z_f\|_{C^{0,\alpha}(A)}
    \le
    c\|df\|_{C^{1,\alpha}(A)},
\]
and similarly for \(h\).

For the quadratic estimate, Taylor's formula in the fiber variables gives
\begin{align}\label{eq: def-B-bilinear}
    G(x,z)=B(x,z)[z,z],
\end{align}
where
\begin{align}\label{eq: B-as-integral}
    B(x,z)
    :=
    \int_0^1(1-t)D_z^2G(x,tz)\,dt .
\end{align}
Since \(B\) is smooth and \(\overline A\) is compact, there exists \(c>0\) such that, for every
\(z\) with \(\|z\|_{C^{0,\alpha}(A)}\le \varepsilon\) (choosing $\varepsilon$ smaller if necessary),
\begin{align*}
    \|B(\cdot,z(\cdot))\|_{C^{0,\alpha}(A)}
    \le c.
\end{align*}
Indeed, in a local trivialization
\[
\begin{aligned}
|B(x,z(x))-B(y,z(y))|
&\le
|B(x,z(x))-B(y,z(x))| +
|B(y,z(x))-B(y,z(y))| \\
&\le
c d(x,y)+c|z(x)-z(y)| \\
&\le
c\bigl(1+[z]_{C^{0,\alpha}(A)}\bigr)d(x,y)^\alpha.
\end{aligned}
\]
Since \(\|z\|_{C^{0,\alpha}(A)}\le\varepsilon\), the right-hand side is bounded
by \(c d(x,y)^\alpha\).

Hence, using that
\(C^{0,\alpha}(A)\) is a Banach algebra,

\begin{align}\label{eq: Holder-estimate-Rf}
    \|R(f)\|_{C^{0,\alpha}(A)}
    &=
    \|G(\cdot,z_f(\cdot))\|_{C^{0,\alpha}(A)} =
    \|B(\cdot,z_f(\cdot))[z_f(\cdot),z_f(\cdot)]\|_{C^{0,\alpha}(A)} \\
    \nonumber
    &\le
    c\|z_f\|_{C^{0,\alpha}(A)}^2 \le
    c\|df\|_{C^{1,\alpha}(A)}^2.
\end{align}

For the second estimate, the fundamental theorem of calculus gives
\[
    G(x,z_f)-G(x,z_h)
    =
    L(x,z_f,z_h)[z_f-z_h],
\]
where
\[
    L(x,z_f,z_h)
    :=
    \int_0^1
    D_zG\bigl(x,z_h+t(z_f-z_h)\bigr)\,dt .
\]
Since \(D_zG(x,0)=0\), smoothness of \(G\) gives
\[
    \|L(\cdot,z_f(\cdot),z_h(\cdot))\|_{C^{0,\alpha}(A)}
    \le
    c\bigl(
        \|z_f\|_{C^{0,\alpha}(A)}
        +
        \|z_h\|_{C^{0,\alpha}(A)}
    \bigr).
\]
Therefore,
\[
\begin{aligned}
    \|R(f)-R(h)\|_{C^{0,\alpha}(A)}
    &\le
    c\|L(\cdot,z_f,z_h)\|_{C^{0,\alpha}(A)}
    \|z_f-z_h\|_{C^{0,\alpha}(A)} \\
    &\le
    c\bigl(
        \|z_f\|_{C^{0,\alpha}(A)}
        +
        \|z_h\|_{C^{0,\alpha}(A)}
    \bigr)
    \|z_f-z_h\|_{C^{0,\alpha}(A)} \\
    &\le
    c\bigl(
        \|df\|_{C^{1,\alpha}(A)}
        +
        \|dh\|_{C^{1,\alpha}(A)}
    \bigr)
    \|df-dh\|_{C^{1,\alpha}(A)}.
\end{aligned}
\]
\end{proof}

Next, we apply a rescaling argument to obtain the estimates on the whole cone. 

\begin{proposition}\label{prop: bound-Rf}
There are constants \(c,\varepsilon_0>0\) such that, if
\[
    \tnorm{df}_{1,\alpha;1}\le \varepsilon_0,
\]
then for every \(\mu\in\mathbb R\) we have
\[
    \tnorm{R(f)}_{0,\alpha;\mu}
    \le
    c\tnorm{df}_{1,\alpha;\frac{\mu}{2}+1}^{2}.
\]
\end{proposition}
\begin{proof}
Recall that by Lemma \ref{lem:R-rescaling}
\[
    R(f)\circ\delta_\lambda=R(f_\lambda),
\]
where $f_\lambda$ is given by \eqref{eq: def-f-lambda}.
Thus, 
\[
    [R(f)]_{0,\alpha;0;\lambda}
    \le c\|R(f_\lambda)\|_{C^{0,\alpha}(A)}.
\]
Moreover,
\[
    \|df_\lambda\|_{C^{1,\alpha}(A)}
    \le
    c\lambda^{-1}[df]_{1,\alpha;0;\lambda}.
\]
Hence, if
\[
    \tnorm{df}_{1,\alpha;1}\le \varepsilon_0
\]
and \(\varepsilon_0\) is sufficiently small, Lemma
\ref{lem:fixed-annulus-R} applies to \(f_\lambda\) for every
\(\lambda\in(0,1/2]\): \eqref{eq:fixed-annulus-R-quadratic} implies that
\begin{align*}
    [R(f)]_{0,\alpha;0;\lambda}
    \le
    c\|df_\lambda\|_{C^{1,\alpha}(A)}^2 \le
    c\lambda^{-2}[df]_{1,\alpha;0;\lambda}^2.
\end{align*}
Multiplying by \(\lambda^{-\mu}\) and taking the supremum over \(\lambda\), we
obtain
\[
\begin{aligned}
    \tnorm{R(f)}_{0,\alpha;\mu}
    &=
    \sup_{\lambda\in(0,1/2]}
    \lambda^{-\mu}[R(f)]_{0,\alpha;0;\lambda} \le
    c\sup_{\lambda\in(0,1/2]}
    \lambda^{-\mu-2}[df]_{1,\alpha;0;\lambda}^2 \\
    &=
    c\sup_{\lambda\in(0,1/2]}
    \left(
        \lambda^{-\frac{\mu}{2}-1}[df]_{1,\alpha;0;\lambda}
    \right)^2 =
    c\tnorm{df}_{1,\alpha;\frac{\mu}{2}+1}^{2}.
\end{aligned}
\]
\end{proof}

\begin{proposition}\label{prop: contraction-estimate}
There are constants \(c,\varepsilon_0>0\) such that, if
\[
    \tnorm{df}_{1,\alpha;1},
    \tnorm{dh}_{1,\alpha;1}
    <
    \varepsilon_0,
\]
then for every \(\mu\in\mathbb R\) we have
\[
    \tnorm{R(f)-R(h)}_{0,\alpha;2\mu-2}
    \le
    c\big(
        \tnorm{df}_{1,\alpha;\mu}
        +
        \tnorm{dh}_{1,\alpha;\mu}
    \big)
    \tnorm{df-dh}_{1,\alpha;\mu}.
\]
\end{proposition}
\begin{proof}
For \(\lambda\in(0,1/2]\) and for $f_\lambda, h_\lambda$ as in \eqref{eq: def-f-lambda}, Lemma \ref{lem:R-rescaling} implies
\[
    (R(f)-R(h))\circ\delta_\lambda
    =
    R(f_\lambda)-R(h_\lambda).
\]
Therefore
\[
    [R(f)-R(h)]_{0,\alpha;0;\lambda}
    \le
    c\|R(f_\lambda)-R(h_\lambda)\|_{C^{0,\alpha}(A)}.
\]
Moreover,
\[
    \|df_\lambda\|_{C^{1,\alpha}(A)}
    \le
    c\lambda^{-1}[df]_{1,\alpha; 0;\lambda},
    \qquad
    \|dh_\lambda\|_{C^{1,\alpha}(A)}
    \le
    c\lambda^{-1}[dh]_{1,\alpha;0;\lambda}.
\]
As $\tnorm{df}_{1,\alpha;1},
    \tnorm{dh}_{1,\alpha;1}
    <
    \varepsilon_0$,
we can apply Lemma \ref{lem:fixed-annulus-R} to $f_\lambda$, $h_\lambda$: \eqref{eq:fixed-annulus-R-lipschitz} implies that
\begin{align*}
    [R(f)-R(h)]_{0,\alpha;0;\lambda}
    &\le
    c\bigl(
        \|df_\lambda\|_{C^{1,\alpha}(A)}
        +
        \|dh_\lambda\|_{C^{1,\alpha}(A)}
    \bigr)
    \|df_\lambda-dh_\lambda\|_{C^{1,\alpha}(A)} \\
    &\le
    c\bigl(
        \lambda^{-1}[df]_{1,\alpha;0;\lambda}
        +
        \lambda^{-1}[dh]_{1,\alpha;0;\lambda}
    \bigr)
    \lambda^{-1}[df-dh]_{1,\alpha;0;\lambda}.
\end{align*}
Multiplying by \(\lambda^{-(2\mu-2)}\), we obtain
\[
\begin{aligned}
    \lambda^{-(2\mu-2)}
    [R(f)-R(h)]_{0,\alpha;0;\lambda}
    &\le
    c\bigl(
        \lambda^{-\mu}[df]_{1,\alpha;0;\lambda}
        +
        \lambda^{-\mu}[dh]_{1,\alpha;0;\lambda}
    \bigr)
    \lambda^{-\mu}[df-dh]_{1,\alpha;0;\lambda}.
\end{aligned}
\]
Taking the supremum over \(\lambda\in(0,1/2]\), we conclude that
\[
    \tnorm{R(f)-R(h)}_{0,\alpha;2\mu-2}
    \le
    c\big(
        \tnorm{df}_{1,\alpha;\mu}
        +
        \tnorm{dh}_{1,\alpha;\mu}
    \big)
    \tnorm{df-dh}_{1,\alpha;\mu}.
\]
\end{proof}

Observe that, if $u,w\in C^{2,\alpha}_{\text{loc}}(C_{1})$ with $\tnorm{du}_{1,\alpha;\nu/2},\tnorm{dw}_{1,\alpha;\nu/2}<\infty$, then by
Proposition~\ref{prop: contraction-estimate} (with $\mu=\nu/2$) followed by \eqref{eq: est-u-CHS}, we have 
\begin{equation*}
\tnorm{\mathcal{P}_\psi(u)-\mathcal{P}_\psi(w)}_{2,\alpha; \nu}
\le c\tnorm{R(u)-R(w)}_{0,\alpha; \nu-2}
\le c(\tnorm {du}_{1,\alpha;\nu/2}+\tnorm{dw}_{1,\alpha;\nu/2})\tnorm{du-dw}_{1,\alpha;\nu/2}.
\end{equation*}
By \eqref{eq: est-fist-second-der} and \eqref{eq: comparison-exponents} (since $\nu > 2$), we have
\begin{equation}\label{eq:du-est}
\tnorm{d\mathcal{P}_\psi(u)-d\mathcal{P}_\psi(w)}_{1,\alpha;\nu-1}
\le c(\tnorm {du}_{1,\alpha;\nu-1}+\tnorm{dw}_{1,\alpha;\nu-1})\tnorm{du-dw}_{1,\alpha;\nu-1}.
\end{equation}

\subsection{The fixed point argument}\label{ssec: fixed-point-argument}

We can now give a proof of Theorem \ref{thm: SL-CHS}. 
\begin{proof}[Proof of Theorem \ref{thm: SL-CHS}]
    Let us first identify an appropriate complete metric space so that we may apply the Contraction Mapping Principle\footnote{The Contraction Mapping Principle says that, if $\mathcal{P}$ is a contraction map on a non-empty complete metric space $(X, d)$, then $\mathcal{P}$ has a unique fixed point \cite[Theorem 9.23]{Ru76}.} to $\mathcal{P}_\psi$. For $\psi\in C^{2,\alpha}(\Sigma)$ and $\nu \in (2,\infty) \setminus \Lambda_\Sigma$, let
\begin{align*}
X := X(\nu,\psi) = \Bigl\{& f \in C^{2,\alpha, \nu}(C_{1}): \Pi_\nu f=\Pi_\nu \psi\text{ on }\Sigma \Bigr\}.
\end{align*}
Then $X$ is a closed subset of $C^{2,\alpha,\nu}(C_1)$. Define a map $d_X: X \times X \rightarrow [0,\infty)$ by
\begin{align*}
    d_X(f,g) := \tnorm{df - dg}_{1,\alpha;\nu-1}.
\end{align*}
Clearly, $d_X$ is symmetric and satisfies the triangle inequality. If $d_X(f,g) = 0$, then $df- dg = 0$, implying $f - g = c_0 \in C^{2,\alpha,\nu}(C_1)$. Since $\nu > 2$, we know 
\[\lim_{r \rightarrow 0^+} (f-g)(r,y) = 0.\]
This forces $c_0 = 0$ implying $(X,d_X)$ is a metric space. We claim that $d_X$ is equivalent to the metric induced by
$\tnorm{\cdot}_{2,\alpha;\nu}$. 

First, by
Lemma~\ref{lem:derivative-weight},
\begin{align}\label{eq: equivalence-norms-1}
\tnorm{dh}_{1,\alpha;\nu-1}
\leq
c\tnorm{h}_{2,\alpha;\nu}
\qquad
\text{for all } h\in C^{2,\alpha,\nu}(C_1).
\end{align}
Conversely, since every $h\in C^{2,\alpha,\nu}(C_1)$ satisfies
$h\to0$ at the cone tip, for $(r,y)\in C_1\simeq (0,1]\times\Sigma$
we have
\begin{align*}
|h(r,y)|
&=
\left|\int_0^r \partial_s h(s,y)\,ds\right|  \leq
\int_0^r |\partial_s h(s,y)|\,ds \leq
c\tnorm{dh}_{1,\alpha;\nu-1}\int_0^r s^{\nu-1}\,ds  \leq
c r^\nu \tnorm{dh}_{1,\alpha;\nu-1}.
\end{align*}
Thus, the zeroth-order weighted part of $\tnorm{h}_{2,\alpha;\nu}$ is
controlled by $\tnorm{dh}_{1,\alpha;\nu-1}$. The remaining first- and
second-order weighted terms, together with the Hölder seminorms, are
controlled directly by the $C^{1,\alpha,\nu-1}$-norm of $dh$. Hence,
\begin{align}\label{eq: equivalence-norms-2}
\tnorm{h}_{2,\alpha;\nu}
\leq
c\tnorm{dh}_{1,\alpha;\nu-1}.
\end{align}

Applying \eqref{eq: equivalence-norms-1} and \eqref{eq: equivalence-norms-2} to $h=f-g$, we obtain
\begin{align*}
c^{-1}\tnorm{f-g}_{2,\alpha;\nu}
\leq
d_X(f,g)
\leq
c\tnorm{f-g}_{2,\alpha;\nu}.
\end{align*}
Since $X$ is closed in the Banach space $C^{2,\alpha,\nu}(C_1)$, it is
complete with respect to the metric induced by
$\tnorm{\cdot}_{2,\alpha;\nu}$. The equivalence above therefore implies
that $(X,d_X)$ is complete.

Fix $\varepsilon\in (0,\varepsilon_0)$ and consider the closed subset
\[
B_\varepsilon(X) :=\{f\in X:\ \tnorm{df}_{1,\alpha;\nu-1}\le \varepsilon\} \subset X.
\]
Then the subspace $(B_\varepsilon(X), d_X)$ of $(X, d_X)$ is complete. We will show that $\mathcal{P}_\psi$ is a contraction map on $B_\varepsilon(X)$. The proof of this claim is broken into two steps. \\

\noindent
\textbf{Step 1:} \textit{For $\varepsilon>0$ sufficiently small, $\mathcal P_\psi$ maps $B_\varepsilon(X)$ to itself.} \\

Let $u\in B_\varepsilon(X)$ and let $v=\mathcal P_\psi(u)$.  By \eqref{eq: est-u-CHS} and 
Proposition~\ref{prop: bound-Rf} (with $\mu=\nu-2$), we have
\[
\tnorm{v}_{2,\alpha;\nu}
\le c\Bigl(\tnorm{R(u)}_{0,\alpha;\nu-2}+\|\Pi_\nu \psi\|_{C^{2,\alpha}(\Sigma)}\Bigr)
\le c\Bigl(\tnorm{du}_{1,\alpha;\nu/2}^2+\|\Pi_\nu \psi\|_{C^{2,\alpha}(\Sigma)}\Bigr).
\]
As $\nu > 2$, by \eqref{eq: comparison-exponents} $\tnorm{du}_{1,\alpha;\nu/2}\leq c\tnorm{du}_{1,\alpha;\nu-1}$. Thus, by \eqref{eq: est-fist-second-der} 
\[
\tnorm{dv}_{1,\alpha;\nu-1}\le c\tnorm{v}_{2,\alpha;\nu}
\le c\Bigl(\varepsilon^2+\|\Pi_\nu \psi\|_{C^{2,\alpha}(\Sigma)}\Bigr).
\]
Note that, since the eigenfunctions $\phi_j$ of $\Delta_\Sigma$ are smooth, $\lVert \Pi_\nu\psi\rVert_{C^{2,\alpha}(\Sigma)}\leq c\lVert\psi\rVert_{C^{2,\alpha}(\Sigma)} $ for some constant $c(\Sigma,\nu)$ depending only on $\Sigma$ and $\nu$.
Therefore, choosing $\varepsilon>0$ and $\psi$ so that
\[
c(\varepsilon^2+c(\Sigma,\nu)\|\psi\|_{C^{2,\alpha}(\Sigma)})\le \varepsilon,
\]
we obtain $\mathcal P_\psi(B_\varepsilon(X))\subset B_\varepsilon(X)$. \\

\noindent \textbf{Step 2:} \textit{$\mathcal {P}_\psi$ is a contraction on $B_\varepsilon(X)$.} \\

If $u,w\in B_\varepsilon(X)$, then \eqref{eq:du-est} (together with  Lemma~\ref{lem:derivative-weight} and the fact that $\nu > 2$) gives
\[d_X(\mathcal  P_\psi(u), \mathcal P_\psi(w))
\le c(\tnorm{du}_{1,\alpha;\nu-1}+\tnorm{dw}_{1,\alpha;\nu-1})d_X(u,w)
\le 2c\varepsilon\,d_X(u,w).
\]
Thus, if $\varepsilon>0$ is chosen so that $2c\varepsilon<1$, $\mathcal{P}_\psi$ is a contraction on $B_\varepsilon(X)$. By the Contraction Mapping Principle, $\mathcal P_\psi$ has a unique fixed point $u\in B_\varepsilon(X)$. From the definition of $\mathcal{P}_\psi$, we have $\Delta u=-R(u)$ implying $\theta(u)=0$. Hence, $L_{du}$ is special Lagrangian.\\
Since $u\in X$, $\Pi_\nu(u\vert_\Sigma-\psi)=0$, therefore $u\vert_\Sigma=\psi+w$ for some $w\in W:=\ker \Pi_\nu$.
\end{proof}
\begin{remark}
Note that if $u \in B_\varepsilon(X)$, then $du$ decays like $O(r^{\nu-1})$ for $\nu - 1 > 1$. This guarantees that the solution submanifold $L_{du}$ is conically singular, modeled on the original cone $C_1$. 
\end{remark}

\section{A special Lagrangian bridge principle}
We prove a bridge principle for special Lagrangian cones in the spirit of \cite[Theorem 4.1]{Sm93}. Theorem \ref{thm: main-theorem} will then follow from the bridge principle after suitable Lagrangian bridges are built. Without loss of generality, we will assume all cones are ``unit length", i.e. they are truncated at distance one from the vertex. The reader should keep \cite{Sm87, Sm89, Sm93, Dim25} close by, as we will refer to them regularly.

Let $C_1(\Sigma^1), \ldots, C_1(\Sigma^N)$ be truncated special Lagrangian cones in $\C^m$ having Lagrangian angle $\theta_0$. Suppose that the $C_1(\Sigma^i)$ have been attached smoothly along their boundaries by Lagrangian bridges of width $O(\varepsilon)$ to form a new Lagrangian submanifold $L^\varepsilon$ with angle close to $\theta_0$, called the \emph{approximate solution}. The bridges joining the cones are called \emph{$\varepsilon$-bridges}, and will be denoted $\mathscr{B}_{\varepsilon}^l$ ($l = 1,\ldots, I$). With this notation, $L^\varepsilon$ can be written 
\[L^\varepsilon =\bigcup_{i=1}^N C_1(\Sigma^i) \cup \bigcup_{l = 1}^I \mathscr{B}_{\varepsilon}^l,\]
where $\{\mathscr{B}_{\varepsilon}^l\}_{l = 1}^I$ is any finite collection of $\varepsilon$-bridges joining the cones. We will set $\mathscr{B}_\varepsilon := \cup_{l=1}^I \mathscr{B}_{\varepsilon}^l$. Our goal is to perturb $L^\varepsilon$ in order to make it special Lagrangian.

For any $u \in C^{2,\alpha,\nu}(L^\varepsilon)$, let
\begin{align*}
    \Gamma^\varepsilon_{du}:=\{(x,du(x))\in T^\ast L^\varepsilon\vert x\in L^\varepsilon\}. 
\end{align*}
As $du$ is closed, $\Gamma^\varepsilon_{du}$ is a Lagrangian section of $T^\ast L^\varepsilon$. By Proposition \ref{prop:weinstein-approximate-solutions}, a neighborhood of the zero section in $T^\ast L^\varepsilon$ can be identified with a neighborhood of $L^\varepsilon$ (regarded as a subset of its normal bundle) by means of a symplectomorphism $\Psi$. Therefore, if $\tnorm{du}_{1, \alpha;\nu-1}$ is sufficiently small ($\nu > 2$), then $\Psi(\Gamma^\varepsilon_{du})$ is a Lagrangian submanifold of $\mathbb{C}^m$. We denote such submanifold by $L^\varepsilon_{du}$.
\begin{remark}\label{rmk: embeddedness}
    Note that if the truncated cones $C^1,...,C^N$ are disjoint, then the $\varepsilon$-bridges can be constructed in such a way that $L^\varepsilon$ is embedded. As $\Gamma^\varepsilon_{du}$ is a section in $T^\ast L^\varepsilon$ and $\Psi$ is a diffeomorphism onto its image, we deduce that in this case $L^\varepsilon_{du}$ is also embedded.
\end{remark}
Let $\Theta(du)(x)$ represent the pointwise Lagrangian angle of $L_{du}^\varepsilon$. Define the \emph{Lagrangian angle defect} $\Theta_0^\varepsilon$ by
\begin{equation}
    \Theta_0^\varepsilon(x) := \theta_{L^\varepsilon}(x) - \theta_0 \text{ in } L^\varepsilon,
\end{equation}
where $\theta_{L^\varepsilon}$ is the Lagrangian angle of $L^\varepsilon$. Then the condition that $L_{du}^\varepsilon$ is special Lagrangian with angle $\theta_0$ is equivalent to the condition that $du$ and $u$ solve
\begin{align}\label{eq: PDE-for-u}
    \Theta(du) = \theta_0 \text{ in }L^\varepsilon \Leftrightarrow \Delta_{L^\varepsilon}u=-\Theta_{0}^\varepsilon-R(u) \text{ in } L^\varepsilon,
\end{align}
where $R(u)$ is defined as in \eqref{eq: def-rest-term-angle}. Note that the argument of Proposition \ref{prop: linearization} remains true in this setting, thus $R(u)$ satisfies \eqref{eq: properties-of-R}.
Going forward, we will drop the $\varepsilon$ to write $\Delta$ and $\Theta_0$, and will denote the right-hand side of the second equation in \eqref{eq: PDE-for-u} by $f(u)$.

For any $\varepsilon > 0$, we define \[\Lambda_{L^\varepsilon} := \bigcup_{i = 1}^N \Lambda_{\Sigma^i}, \]
where $\Lambda_{\Sigma}$ is defined as in Appendix \ref{ssec: existence-uniqueness} (note that $\Lambda_{L^\varepsilon}$ only depends on the spectral properties of the links $\Sigma^i$). We solve \eqref{eq: PDE-for-u} as follows: for $\psi\in C^{2,\alpha}(\partial L^\varepsilon)$ supported away from the bridges, for $u\in C^{2,\alpha, \nu}(L^\varepsilon)$ in a suitable subspace to be determined (see \eqref{Kepsilon}), and for $\nu \in (2,\infty) \setminus \Lambda_{L^\varepsilon}$, define $\mathcal{P}_\psi(u)$ to be the unique solution $v$ of
\begin{align}
\begin{cases}
    \Delta v=f(u) & \text{in } L^\varepsilon,\\
    \Pi_\nu^\varepsilon v=\Pi_\nu^\varepsilon\psi & \text{on } \partial L^\varepsilon, \label{dp}
\end{cases}
\end{align}
guaranteed in Remark \ref{rmk: CHS}, where $\Pi^\varepsilon_\nu$ is the projection onto the $L^2$-orthogonal complement of finitely many elements in $L^2(\partial L^\varepsilon)$ (see \eqref{eq: epsilon-projection} for the precise definition).
The strategy is to show that for small boundary data relative to $\varepsilon$ and $\varepsilon$ sufficiently small, one can find a fixed point of the operator $\mathcal{P}_\psi$, i.e. a solution of \eqref{eq: PDE-for-u}, using the Schauder Fixed Point Theorem\footnote{The Schauder Fixed Point Theorem says that if $K$ is a non-empty convex compact subset of a Banach space $X$, and if $\mathcal{P} : K \rightarrow K$ is a
continuous map, then $\mathcal{P}$ has a fixed point \cite[Theorem 11.1]{GT01}.}. The special Lagrangian bridge principle can be stated quantitatively as follows.
\begin{theorem}\label{bridgeprinciple}
For any $\nu$ sufficiently large with $\nu, \nu+2\notin \Lambda_{L^\varepsilon}$ ($\nu > 2$), there exists a number $\varepsilon_0 > 0$ (depending on $\nu$) and a positive function $\overline{\varepsilon}$ defined on $(0, \varepsilon_0)$ with the following property: for any sequence $\{\varepsilon_k\}_{k\in \mathbb{N}}$ in $(0,\varepsilon_0)$ tending to zero, there exists a subsequence (not relabeled) such that for any $\psi \in C^{2,\alpha}(\partial L^{\varepsilon_k})$ supported away from the $\varepsilon_k$-bridges with $\|\psi\|_{C^{2,\alpha}} \leq \overline{\varepsilon}(\varepsilon_k)$, there is $u \in C^{2,\alpha,\nu}(L^{\varepsilon_k})$ solving $\mathcal{P}_\psi(u)=u$. Furthermore, $\tnorm{du}_{1,\alpha;\nu-1} \rightarrow 0$ as $\varepsilon_k \rightarrow 0$.
\end{theorem}
\begin{remark}[Uniqueness]
    Let $u \in C^{2,\alpha,\nu}(L^\varepsilon)$ be such that $\mathcal{P}_\psi(u)=u$. Since $L_{du}^\varepsilon$ is special Lagrangian with angle $\theta_0$, it is calibrated by $\Re(e^{-i\theta_0}\Omega)$. Hence, \cite[Theorem 1.2]{DL225} shows that $L_{du}^\varepsilon$ is the (global) unique solution to the Plateau problem in the class of integral currents in $\C^m$ spanned by its boundary. However, we cannot guarantee that $u$ uniquely solves the fixed point problem for $\mathcal{P}_\psi$ in $C^{2,\alpha,\nu}(L^\varepsilon)$ since only the projection of the boundary data onto a fixed finite codimensional subspace is prescribed a priori, while the remaining finite dimensional boundary component depends on $u$.
\end{remark}

Of course, we cannot expect to work with arbitrary $\varepsilon$-bridges. In order to run the fixed point argument, we require that
\begin{equation}
        c^{-1}\varepsilon^{m-1} \leq \vol (\mathscr{B}_\varepsilon) \leq c \varepsilon^{m-1}, \label{bridge1}
\end{equation}
where $c$ is a constant independent of $\varepsilon$. Condition \eqref{bridge1} helps ensure that bridge effects introduced during gluing are negligible for small $\varepsilon$. Our constructed bridges will satisfy
\begin{equation}
    \|A_{\partial \mathscr{B}_\varepsilon}\|_{C^0(\partial \mathscr{B}_\varepsilon)} \leq c\varepsilon^{-1}, \label{bridge2}
\end{equation}
where $A_{\partial \mathscr{B}_\varepsilon}$ is the second fundamental form of $\partial \mathscr{B}_\varepsilon$. The boundary estimate \eqref{bridge2} is a result of a smoothing procedure near the patching regions where the bridges meet the cone boundaries, which introduces regions of large curvature relative to $\varepsilon$ (see \cite[Section 2]{Sm87}).

Notice that, in the bridge case, the term $\Theta_0$ may not be zero. We will assume 
\begin{equation}
    \|\Theta_0\|_{C^0(L^\varepsilon)} \leq c \varepsilon^3, \text{ } \|\nabla \Theta_0\|_{C^0(L^\varepsilon)} \leq c\varepsilon^2, \text{ and }  \Big(\int_{L^\varepsilon} |\Theta_0|^{p} \, d\vol_{L^\varepsilon}\Big)^{\frac{1}{p}} \leq c \varepsilon^{\frac{m - 1 + 3p}{p}} \text{ for each } p \geq 1.\label{bridgeest1}
\end{equation}
Condition \eqref{bridgeest1} says that $L^\varepsilon$ is ``approximately special Lagrangian" with angle $\theta_0$, justifying the term ``approximate solution." Notice that, in contrast to the assumed estimates for the mean curvature in \cite{Sm89, Sm93}, we require a $C^1$ estimate for $\Theta_0$ in terms of $\varepsilon$. This additional hypothesis will help us complete the fixed point argument in the special Lagrangian setting. In Section \ref{section: approxsoln}, we show that $L^\varepsilon$ can always be constructed satisfying \eqref{bridge1}, \eqref{bridge2}, and \eqref{bridgeest1}. This is achieved using flat Lagrangian bridges based on those in \cite[Section 7.2]{Dim25}.

The chapter is organized as follows. In Subsection \ref{section: approxsoln}, we construct Lagrangian approximate solutions with flat bridges satisfying \eqref{bridge1}, \eqref{bridge2}, and \eqref{bridgeest1}. Subsection \ref{section: schauder} consists of the supremum and Schauder estimates (both local and global) we will use to run the fixed point argument. Finally, Subsection \ref{ssec: fixed-point-argument2} is dedicated to the fixed point argument.

\subsection{Construction of approximate solutions}\label{section: approxsoln}
To extend Smale's bridge principle to special Lagrangian cones, we need to demonstrate that we can build suitable Lagrangian bridges. It turns out that the flattening procedure developed in \cite[Section 7.2]{Dim25} does the trick, since the tangent plane to a special Lagrangian cone with Lagrangian angle $\theta_0$ is itself special Lagrangian with angle $\theta_0$. For Smale's original bridge constructions, see \cite[Section 2]{Sm87}. To help guide the reader, we begin with a short sketch of the basic bridge construction, focusing on the case of two cones.

Given two truncated special Lagrangian cones with Lagrangian angle $\theta_0$, we position them so that they are disjoint, they lie tangent to the same special Lagrangian plane (with angle $\theta_0$), the ray of tangency of each cone with the plane coincide, and the cones open toward each other. We then extend the cones via dilation in an $\varepsilon$-neighborhood of the points on the links of the cones making contact with the shared tangent plane, and carefully flatten the extended portion onto the plane. As in \cite[Section 7.2]{Dim25} when $n = 3$, this allows us to join the cones using a thin strip (i.e. a \emph{race-car track}) lying in the plane which is a fattening of the ray of tangency along its orthogonal complement in the plane.

We now demonstrate how to rigorously perform the flattening. Fix a truncated special Lagrangian cone $C_1 := C_1(\Sigma) \subset \C^m$. Using $\C^m \simeq \R^{2m}$, up to a translation we can view $C_1$ as an $m$-dimensional cone in $\R^{2m}$ with vertex at the origin. Choose coordinates $(x,z) \in \R^{2m} \simeq \R^m \times \R^m$ so that $C_1$ lies tangent to the $x$-subspace, $x := (x_1,\ldots, x_m)$, with the $x_m$-axis coinciding with the ray of tangency of $C_1$ with the $x$-subspace. 

Set $x := (\tilde{x}, x_m) \in \R^{m-1} \times \R$.
Note that, since $C_1$ is Lagrangian and its tangent plane at $p:=(0,...,0,1)$ is $\mathbb{R}^m$, there is a 2-homogeneous map $u$ defined in a neighborhood of $p$ in $\mathbb{R}^m$ such that the graph of $\nabla u$ coincides with $C(\Sigma)$ (the extended cone) around $p$. Then $u$ must have the form 
\begin{equation*}
    u: \{|\tilde{x}| \leq \varepsilon\} \times (1-\delta, 1+ \delta) \rightarrow \R, \qquad u(\tilde{x}, x_m) := (x_m)^2 v\Big(\frac{\tilde{x}}{{x_m}}\Big),
\end{equation*}
where $\delta$ is independent of $\varepsilon$. Then, near $p$, the graph of $\nabla u$ is a thin wedge $\mathscr{W}$ of $C$ centered about the $x_m$-axis.

We may assume without loss of generality that $u(p) = 0$ since $u$ is a potential. By construction, we have
\begin{equation}
    v(0) = 0, \text{ } \nabla v(0) = 0, \text{ and } D^2 v(0) = 0. \label{3zeros}
\end{equation}
One can check by direct computation using \eqref{3zeros} that 
\begin{equation*}
    u(0,x_m) = 0, \text{ } 
    \nabla  u(0, x_m) = 0, \text{ and } \partial_{x_i}\partial_{x_j}u(0,x_m)=0 \text{ for } i,j\in \{1,...,m-1\}.
\end{equation*}
As a consequence, we deduce
\begin{align*}
    u(x) &\sim O(|\tilde{x}|^3) \\
    \nabla u(x) &\sim O(|\tilde{x}|^2) \\
    \partial_{x_m} u(x) &\sim O(|\tilde{x}|^3)\\
    D^2 u(x) &\sim O(|\tilde{x}|) \\
    \partial_{x_i x_m}^2 u(x) &\sim O(|\tilde{x}|^2) \text{ for } i = 1,\ldots, m-1, \\ 
    \partial_{x_m x_m}^2 u(x) &\sim O(|\tilde{x}|^3).
\end{align*}
Furthermore, $u$ solves the special Lagrangian equation with zero right hand side since its gradient graph $\mathscr{W}$ is minimal and its tangent plane, the $x$-subspace, has Lagrangian angle zero. Expanding as a Taylor series for $\varepsilon$ small gives
\[0 = \arctan(D^2u(x)) = \sum_{j = 1}^m \arctan(\lambda_j(D^2u(x))) =  \Delta u(x) + O(|D^2 u(x)|^3),\]
where the $\lambda_j(D^2u(x))$ represent the eigenvalues of $D^2 u(x)$. Re-arranging terms and using that $D^2u(x)\sim O(\lvert\tilde x\rvert)$ shows that
\[\Delta u(x) \sim O(|\tilde{x}|^3).\]

Let $\phi : \R \rightarrow \R$ be a smooth function that is $1$ on $[1-\delta, 1]$ and decreases from $1$ to $0$ on $[1, 1+\delta]$. Replace $\mathscr{W}$ by the graph of $w: = \nabla(\phi(x_m)u(x))$ and let $\mathscr{P}_\varepsilon$ be the graph of the restriction of $w$ in the $x_m$ variable to $[1, 1+\delta]$ (i.e. the \emph{patching region}). Then $\partial(C_1 \cup \mathscr{P}_\varepsilon)$ is smooth everywhere, except possibly at points on $\Sigma \cap \partial \mathscr{P}_\varepsilon$. After further restriction of the domain of $w$, we may assume $C_1\cup \mathscr{P}_\varepsilon$ is smooth up to the boundary (see \cite[Section 2]{Sm87} and \cite[Section 7.2]{Dim25}). We have
\[\Delta (\phi(x_m)u(x)) = \phi(x_m)\Delta u(x) + 2\phi^{ \prime}(x_m) \partial_{x_m} u(x) + \phi^{\prime \prime}(x_m)u(x). \]
Substituting the estimates derived for $u$ and its derivatives in the previous paragraph into the expression above, we find
\[\Delta(\phi(x_m)u(x)) \sim O(|\tilde{x}|^3).\]

Define
\[
\Theta((\phi u)(x)) = \sum_{j = 1}^m \arctan(\lambda_j(D^2(\phi u)(x))) =  \Delta (\phi u)(x) + O(|D^2 (\phi u)(x)|^3),\]
to be the Lagrangian angle of the graph of $\phi u$. From the previous computations, $D^2((\phi u)(x)) \sim O(|\tilde{x}|)$. We deduce that 
\begin{equation}
    \Theta((\phi u)(x)) \sim O(|\tilde{x}|^3) \text{ on } \mathscr{P}_\varepsilon. \label{bridgeest}
\end{equation}
We can now derive an $L^p$ estimate for $\Theta(\phi u)$ on $\mathscr{P}_\varepsilon$ using the pointwise bound on $\Theta(\phi u)$ and the fact that the $m$-dimensional volume of $\mathscr{P}_\varepsilon$ and the domain of $u$ are comparable:
\begin{equation}
    \Big(\int_{\mathscr{P}_\varepsilon} |\Theta(\phi u)|^p \Big)^{\frac{1}{p}} \leq c \varepsilon^{\frac{m-1+3p}{p}}. \label{slagest}
\end{equation}
The estimates \eqref{bridgeest} and \eqref{slagest} will imply the $C^0$ and $L^p$ bounds on $\Theta$ in \eqref{bridgeest1}.

The next step is to prove the $C^1$ bound for $\Theta$ in \ref{bridgeest1}. Let $H$ be the mean curvature for $\mathscr{P}_\varepsilon$. Since $H = J\nabla \Theta$ (see Remark \ref{rmk: mean-curvature-lagrangian-angle}), it suffices to prove the desired bound for $H$ replacing $\nabla \Theta$. The wedge $\mathscr{W}$ is minimal (again by Remark \ref{rmk: mean-curvature-lagrangian-angle}) implying $\nabla u$ solves the minimal surface system \cite[Chapter 11]{GM12}
\[g_{C}^{ij}(x)\partial_{x_ix_j}^2 u_{x^\beta}(x) = 0 \text{ for each } \beta = 1,\ldots, m\]
where $g_{C} = I + (D\nabla u)^T(D\nabla u)$ is the natural metric on the cone. We can then expand the left-hand side as a Taylor series to obtain
\[0 = \Delta \nabla u(x) + O(|D^2u(x)|^2) \Rightarrow \Delta \nabla u(x) \sim O(|\tilde{x}|^2).\]
Here, we used the fact that $D^3u$ is bounded.
Now, $\mathscr{P}_\varepsilon$ is parameterized by $\psi(x) := (x, w(x))$ where
\[w(x) = (\phi(x_m)\nabla_{\tilde{x}}u(x), \phi^\prime(x_m) u(x) + \phi(x_m) \partial_{x_m}u(x)),\]
and $H$ is defined by the formula \cite[eq. (11.7), p. 298]{GM12}
\[H(x) = \big(g^{ij}(x) \partial_{x_i x_j} \psi(x)\big)^{\perp},\]
where $g := I + Dw^T Dw$ is the metric on $\mathscr{P}_\varepsilon$. Since $\partial_{x_i x_j}(x^\beta) = 0$ for each $i,j$ and each $\beta = 1,\ldots, m$, it suffices to estimate $g^{ij}\partial_{x_i x_j} w^\beta$ for each $\beta$. Expanding as a Taylor series once more, we find
\[g^{ij}(x)\partial_{x_i x_j} w^\beta(x) = \Delta w^\beta(x) + O(|Dw(x)|^2).\]
Due to the estimates for $u$ and its derivatives, we know $O(|Dw(x)|^2) = O(|\tilde{x}|^2)$. Thus, we only need to estimate the Laplacian term. 

Direct computation shows that for $\beta = 1,\ldots, m-1$ we have 
\[\Delta w^\beta(x) = \phi^{\prime \prime}(x_m)\partial_{x^\beta}u(x) + 2\phi^\prime(x_m) \partial_{x^\beta x_m}u(x) + \phi(x_m) \Delta \partial_{x^\beta}u(x).\]
Each term on the right-hand side above is $O(|\tilde{x}|^2)$ by previous estimates, implying the same is true for $\Delta w^\beta$ when $\beta = 1,\ldots, m-1$. When $\beta = m$, we get 
\begin{align*}
    \Delta w^m(x) &= \phi^\prime(x_m) \Delta u(x) + 2 \phi^{\prime \prime}(x_m)\partial_{x_m} u(x) + \phi^{\prime \prime \prime}(x_m)u(x) \\
    &\qquad+ \phi(x_m) \Delta \partial_{x_m}u(x) + 2\phi^\prime(x_m) \partial_{x_m}^2 u(x) + \phi^{\prime \prime}(x_m) \partial_{x_m} u(x).
\end{align*}
Again, each term  is at least $O(|\tilde{x}|^2)$ by prior estimates. We conclude $\Delta w(x) \sim O(|\tilde{x}|^2)$. In particular, the desired estimate holds for $H$ on $\mathscr{P}_\varepsilon$, and thus for $\nabla \Theta$. 

Suppose now that we are given two truncated special Lagrangian cones $C_1(\Sigma^1)$ and $C_1(\Sigma^2)$ in $\C^m$ with a common tangent plane $T$. Then we can lie them both tangent to the same $m$-plane, flatten them onto the plane as above, and connect them by a race-car track $\mathscr{T}_\varepsilon$ in $T$ of width $\varepsilon$ as described in the second paragraph in this section. The plane of tangency is special Lagrangian with the same Lagrangian angle as the cones. Then $\mathscr{B}_\varepsilon := \mathscr{P}_\varepsilon \cup \mathscr{T}_\varepsilon$ is an $\varepsilon$-bridge from $C_1(\Sigma^1)$ to $C_1(\Sigma^2)$. Since $\Theta \equiv 0 \equiv H$ on the cones and on $\mathscr{T}_\varepsilon$, the estimates \eqref{slagest} hold with $L^\varepsilon$ replacing $\mathscr{P}_\varepsilon$. Proceeding by induction, it is clear that this construction can be done for any finite list of special Lagrangian cones $C_1(\Sigma^1), \ldots, C_1(\Sigma^N)$ pairwise sharing a common tangent plane. 

For a simple explicit construction, if each of the cones $C_1(\Sigma^i)$ ($i = 1,\ldots, N$) lie tangent to the same $m$-plane, then we can flatten a small portion of their boundaries onto the plane as above and join the cones by race-car tracks meeting in a roundabout of radius $\varepsilon$ having $N$ exits. Since $\Theta \equiv 0$ and $H \equiv 0$ outside of the patching region $\mathscr{P}_\varepsilon := \cup_{l = 1}^N \mathscr{P}_{\varepsilon}^l$ on the approximate solution $L^\varepsilon$ constructed in this manner, the bounds on $\Theta$ and $H$ do not change. Finally, by restriction we can construct the approximate solutions so that $L^{\varepsilon_1} \subset L^{\varepsilon_2}$ whenever $\varepsilon_1 < \varepsilon_2$ as in \cite{Sm87, Dim25} --- hence, the conical Weinstein neighborhood does not vary for $\varepsilon < \varepsilon_0$.

\begin{remark}
    Note that Smale's bridges in \cite{Sm87} and \cite{Sm89} do not need to be flat, while in \cite{Dim25} the bridges are only required to be flat near the patching region in low dimensions and when building graphs in dimensions $m \ge 4$ in order to avoid excess curvature concentration in $\mathscr{P}_\varepsilon$ and preserve graphicality. The present construction is most similar to the case $n =3 $ in \cite[Section 7.2]{Dim25}.
\end{remark}

\subsection{Schauder estimates}\label{section: schauder}
We record the local and global estimates needed to run the bridge principle algorithm in the special Lagrangian setting. The local estimates are standard and can be found in \cite[Proposition 3.1]{Sm89} or \cite[Proposition 4.1]{Sm93}, which are taken from \cite{GT01}.
\begin{proposition}[Local Estimates]\label{prop: localest}
    Let $K \subset L^\varepsilon \setminus \sing(L^\varepsilon)$ be a compact subset and suppose $u \in C^{2,\alpha}(K)$ solves $\Delta u = f$ in $K$. Assume $B_{2r} \subset K$ is a geodesic ball of radius $2r$. Then there are constants $c(p)$, $c(\alpha)$, and $c$ such that
    \begin{itemize}
        \item[(a)] $\| u\|_{C^0( B_r)} \leq c(p)\big( r^{-\frac{m}{2}}\| u \|_{L^2(B_{2r})} + r^{2-\frac{m}{p}}\|f \|_{L^p(B_{2r})} \big)$, where $p > m/2$.
        \item[(b)] $\| \nabla^k u \|_{C^0( B_{r})} \leq c\big(r^{-k}\|u\|_{C^0( B_{2r})} +  r^{2-k}\|f\|_{C^{0,\alpha}(B_{2r})}\big)$ for $k = 1,2$.
        \item[(c)] $[\nabla^k u]_{(\alpha); B_{r}} \leq c(\alpha)\big(r^{-k-\alpha}\| u \|_{C^0(B_{2r})} + r^{2 - k -\alpha}\|f \|_{C^{0,\alpha}(B_{2r})}\big)$ for $k = 0,1,2$.
    \end{itemize}
\end{proposition}
Both global $C^0$ and Schauder estimates are needed on the truncations $L_\sigma^\varepsilon$. The global $C^0$ estimate follows from the proof of \cite[Theorem 8.16]{GT01} applied to $\Delta$, if the Euclidean Sobolev inequality is replaced by the Michael-Simon inequality \cite{MS73} --- whose constant is independent of $\varepsilon$.
\begin{proposition}[Global $C^0$ bound] \label{c0}
    Suppose $\Delta u = f$ in $L_\sigma^\varepsilon$ for some $u \in C_0^{2}(L_\sigma^\varepsilon)$. Then there is a constant $c:= c(p,\sigma)$ such that, for any $p > m/2$, we have
    \[\|u\|_{C^0(L_\sigma^\varepsilon)} \leq c\|f\|_{L^p(L_\sigma^\varepsilon)}.\]
\end{proposition}

The global Schauder estimates on $L_\sigma^\varepsilon$ we use originally appear in \cite[Lemma 1 \& 2, p. 522]{Sm87}, where they are proved for a rather general class of second-order strongly elliptic principally diagonal linear systems that include $\Delta$ (see pp. 521--522). However, we adopt the version in \cite{Sm89} or \cite{Sm93}. Using \eqref{bridge2}, we deduce that the boundary hypothesis in \cite[(f), p. 521]{Sm87} is met for $L_\sigma^\varepsilon$ allowing us to apply the estimates in our setting. 
\begin{proposition}[Global Schauder estimates]\label{prop: globalschauder}
    Suppose $u \in C^{2,\alpha}(L_\sigma^\varepsilon)$, $u = 0$ on $\partial L_\sigma^\varepsilon$, and $\Delta u = f$ in $L_\sigma^\varepsilon$. Then there is a constant $c: = c(\alpha,\sigma)$ such that 
    \begin{itemize}
        \item[(a)] $\| du\|_{C^{1}(L_\sigma^\varepsilon)} \leq c\big(\varepsilon^{-2}\|u\|_{C^0(L_\sigma^\varepsilon)} +  \|f\|_{C^{0,\alpha}(L_\sigma^\varepsilon)} \big)$.
        \item[(b)] $[\nabla du]_{(\alpha); L_\sigma^\varepsilon} \leq c\big(\varepsilon^{-2-\alpha}\|u\|_{C^0(L_\sigma^\varepsilon)} + \varepsilon^{ - \alpha}\|f\|_{C^{0,\alpha}(L_\sigma^\varepsilon)}\big)$.
    \end{itemize}
\end{proposition}
\begin{proof}
    This is derived directly from \cite[Proposition 3.3]{Sm89} with $\Delta$ replacing $L$. We have only included the higher order estimates, since they are all we need for the proof in our setting.
\end{proof}

To extend the Caffarelli--Hardt--Simon method to special Lagrangian cones, we proved the following pointwise estimates for $R(u)$ (though stated differently).
\begin{proposition}[Pointwise error estimate]\label{prop: pointerror}
    For $\nu > 2$, $\alpha \in (0,1)$, and any $u \in C_{\text{loc}}^{2,\alpha}(L^\varepsilon)$ with $\tnorm{du}_{1,\alpha,\nu-1}$ sufficiently small there is a constant $c := c(\alpha)$ such that 
    \begin{itemize}
        \item[(a)] $|R(u)| \leq c\big(\rho^{-2}|du|^2 + |\nabla du|^2\big)$ pointwise for in $L^\varepsilon$.
        \item[(b)] $[R(u)]_{\alpha;[\sigma,2\sigma]} \leq c \sigma^{-2-\alpha}[du]_{1;0;\sigma}[du]_{1,\alpha;0;\sigma}$ for any $\sigma\in (0,\frac{1}{2}]$.
    \end{itemize}
\end{proposition}
\begin{proof}
    Extend the definition of $G$ in \eqref{eq: def-of-H} to $T^*L^\varepsilon\oplus \operatorname{Sym}^2T^*L^\varepsilon$ in the natural way. Then \eqref{eq: def-B-bilinear} and \eqref{eq: B-as-integral} imply that $(a)$ is satisfied around the bridges. Arguing by rescaling as in the proof of Proposition \ref{prop: bound-Rf}, one verifies that the bound also hold on the cones.
    
    To prove estimate $(b)$, note that the argument leading to \eqref{eq: Holder-estimate-Rf} implies that
    \begin{align*}
        [R(u)]_{(\alpha); A}\leq c\lVert(du,\nabla du)\rVert_{L^\infty(A)}\lVert(du,\nabla du)\rVert_{C^{0,\alpha}(A)}.
    \end{align*}
    Arguing by rescaling as in the proof of Proposition \ref{prop: bound-Rf}, we see that the estimate holds on every annulus.
\end{proof}
\subsection{The fixed point argument}\label{ssec: fixed-point-argument2}
In this section, we prove Theorem \ref{bridgeprinciple}. For clarity of exposition, we present the proof in the case of two ``unit length" special Lagrangian cones $C^1$ and $C^2$ joined by an 
$\varepsilon$-bridge to form a Lagrangian approximate solution $L^\varepsilon$. However, the argument applies to any configuration of truncated special Lagrangian cones which pairwise share common tangent planes.

For $\nu >2$ with $\nu,\nu+2  \notin \Lambda_{L^\varepsilon}$, set $\delta := \delta(\nu)=\nu+\frac{m}{2}-1$.
Let $\nu_0>0$ be such that the kernel of $\Delta$ in $C^{2,\alpha,\nu_0}(C^1\cup C^2)$ is trivial. Note that this remains true for any $\nu>\nu_0$, and that, by \cite[Lemma 2.2]{Sm93}, the corresponding $\delta$ is such that
\begin{align*}
    \Delta: \rho^\delta H^2\cap H^1_0(C^1\cup C^2)\to \rho^{\delta-2}L^2(C^1\cup C^2)
\end{align*}
has empty kernel (note that in \cite{Sm93}, Smale works with the rescaled Laplacian $\rho^2\Delta$, which maps to a different weighted space). From now on, assume that $\nu>\nu_0$.
For $t,\varepsilon>0$ to be determined, define $\mathscr{K}(\varepsilon,t) := \mathscr{K}(\varepsilon, \alpha, \nu, t)$ by
\begin{align}
    \mathscr{K}(\varepsilon, t) = \{u \in  C^{2,\alpha, \nu}(L^\varepsilon) : &\tnorm{du}_{1;\nu- 1} \leq \varepsilon^{2  - t} \label{Kepsilon}, \\
    &\tnorm{du}_{1, \alpha; \nu - 1} \leq \varepsilon^{2 - \alpha - t}, \text{ and } \lVert du\rVert_{1; \delta+1}\leq \varepsilon^{\frac{m+2}{2}-t} \}. \nonumber
\end{align}
Note that any element $u\in \mathscr{K}(\varepsilon, t)$ must tend to zero at the cone tips, therefore $\mathscr{K}(\varepsilon, t)$ is a bounded subset of $C^{2,\alpha,\nu}(L^\varepsilon)$.

For any $\varepsilon>0$, let $K_H^{L^\varepsilon}(-\delta,\Delta)$ denote the kernel of the map
\begin{align*}
    \Delta: \rho^{-\delta}  H^1_0\cap  H^2(L^\varepsilon)\to \rho^{-2-\delta} L^2(L^\varepsilon).
\end{align*}
By \cite[Lemma 4.2]{Sm93}, there are constants $J(\nu)$ and $\varepsilon_0$ such that, for any $\varepsilon\in (0,\varepsilon_0)$, we have
\begin{align*}
    \dim K_H^{L^\varepsilon}(-\delta,\Delta)=J(\nu).
\end{align*}
Let $\{\varepsilon_k\}_{k\in \mathbb{N}}$ be a sequence in $(0,\varepsilon_0)$ tending to zero.
For any $k\in \mathbb{N}$, let $\{w_1^k, ..., w_{J(\nu)}^k\}$ be an orthonormal basis of $K_H^{L^{\varepsilon_k}}(-\delta,\Delta)$ with respect to $\rho^{-\delta} L^2(L^{\varepsilon_k})$.

Observe that, by the proof of Lemmas 4.1 and 4.2 in \cite{Sm93}, given a sequence $\{w_k\}_{k\in \mathbb{N}}$ of elements of $K_H^{L^{\varepsilon_k}}(-\delta,\Delta)$ with $\lVert w_k\rVert_{\rho^{-\delta} L^2(L^{\varepsilon_k})}=1$, one can extract a subsequence (not relabeled) so that the functions $v_k:=w_k\vert_{C^1\cup C^2}$ converge in $\rho^{-\delta'}H^1(C^1\cup C^2)$ (for $\delta'>\delta$) to a function $v\in K_H^{C^1\cup C^2}(-\delta,\Delta)$ with $\lVert v\rVert_{\rho^{-\delta} L^2(C^1\cup C^2)}=1$. Therefore, by a diagonal argument we may pass to a subsequence (not relabeled) and assume that, for any $i\in \{1,..., J(\nu)\}$, the functions $v_i^k:=w_i^k\vert_{C^1\cup C^2}$ converge to a function $v_i\in K_H^{C^1\cup C^2}(-\delta,\Delta)$ as $k \rightarrow \infty$. Moreover, $\{v_1,...v_{J(\nu)}\}$ is an orthonormal basis of $K_H^{C^1\cup C^2}(-\delta,\Delta)$ with respect to $\rho^{-\delta} L^2(C^1\cup C^2)$ (the orthogonality follows from the $L^2$ estimate in a neighborhood of the bridges in the proof of \cite[Lemma 4.1]{Sm93}).

Note that, since $\{v_i^k\}_{k\in \mathbb{N}}$ converges in $\rho^{-\delta'}H^1(C^1\cup C^2)$ and every $v_i^k$ is harmonic, we have that $\partial_\nu v_i^k\to \partial_\nu v_i$ in $H^{-\frac{1}{2}}((\partial C^1\cup \partial C^2)\smallsetminus (B_\eta(p)\cup B_\eta(q)))$. Here, $\eta\in (0,1)$ is small and $p,q$ denote points on $C^1$ and $C^2$, respectively, where the $\varepsilon$-bridge is attached (these correspond to the point $p$ in Subsection \ref{section: approxsoln}).
For $\eta>0$ and $\varepsilon_k$ as above, let
\begin{align*}
    \mathscr{F}^{\varepsilon_k}(\eta)=\{\psi\in C^{2,\alpha}(\partial L^{\varepsilon_k}): \operatorname{supp}(\psi)\subset\partial L^{\varepsilon_k}\smallsetminus(\mathscr{B}_{\varepsilon_k}\cup B_{2\eta}(p)\cup B_{2\eta}(q))\}.
\end{align*}
Define the map
\begin{align*}
    B_\nu^{\varepsilon_k}: \mathscr{F}^{\varepsilon_k}(\eta)\to \mathbb{R}^{J(\nu)},\quad \psi\mapsto \left(\int_{\partial L^{\varepsilon_k}}\psi\cdot\partial_{\nu}w_1^k\,,..., \int_{\partial L^{\varepsilon_k}}\psi\cdot\partial_{\nu}w_{J(\nu)}^k\right),
\end{align*}
and the analogous map $B_\nu^0$ for maps on $\partial C^1\cup \partial C^2$.

We now let $\{e_1,..., e_{J(\nu)}\}$ be such that, for $i\in \{1,..., J(\nu)\}$, we have $B_\nu^0(e_i)=(0,...,1,...,0)$, where the only $1$ is at the $i$-th position. Note that, since $v_i^k\to v_i$ in $\rho^{-\delta'}H^1(C_1\cup C_2)$, we have $\int_{\partial C^1\cup \partial C^2}e_i\cdot \partial_\nu w_j^k=\delta_{ij}+\delta^k$. Here, $\delta_{ij}$ is the Kronecker delta and $\delta^k\to 0$ as $k\to \infty$. Thus, for large enough $k$ the matrix $\textbf{A}^k\in \mathbb{R}^{J(\nu)\times J(\nu)}$ given by
\begin{align*}
    \mathbf{A}^k_{ij}=\int_{\partial C^1\cup \partial C^2}e_i\cdot \partial_\nu w_j^k
\end{align*}
is invertible, and its inverse $\mathbf{B}^k$ has entries bounded by $2$. 

For any $k\in \mathbb{N}$ and $i\in \{1,..., J(\nu)\}$, set $e_i^k:=\sum_j\mathbf{B}_{ij}^ke_j$ (regarded as an element in $\mathscr{F}^{\varepsilon_k}(\eta)$). Note that
    $B_\nu^{\varepsilon_k}(e_i^k)= (0,...,1,...,0)$, where the only $1$ is at the $i$-th position.
As $\mathbf{B}^k$ is uniformly bounded, we have in particular that
\begin{align}\label{eq: uniform-estimate-eik}
    \lVert e_i^k\rVert_{C^{2,\alpha}( \partial C^1\cup \partial C^2)}\leq c
\end{align}
for some constant $c$ independent of $k$.\\
For any $k\in \mathbb{N}$, we also set
\begin{align}\label{eq: epsilon-projection}
    \Pi_\nu^{\varepsilon_k}: \mathscr{F}^{\varepsilon_k}(\eta)\to \ker B_\nu^{\varepsilon_k},\quad \psi\mapsto \psi-\sum_{i=1}^{J(\nu)}\left(\int_{\partial L^{\varepsilon_k}}\psi\cdot \partial_\nu w_i^k\right)e_i^k.
\end{align}
Since $\partial_\nu w_i^k$ are uniformly bounded in $H^{-\frac{1}{2}}(\partial L^{\varepsilon_k}\smallsetminus (\mathscr{B}_{\varepsilon_k}\cup B_{2\eta}(p)\cup B_{2\eta}(q)))$ and \eqref{eq: uniform-estimate-eik} holds, we have
\begin{align}
\label{eq: holder-bound-bdry}
    \lVert \Pi_\nu^{\varepsilon_k}(\psi)\rVert_{C^{2,\alpha}( \partial L^{\varepsilon_k})}\leq c\lVert \psi\rVert_{C^{2,\alpha}(\partial L^{\varepsilon_k})}\quad \forall \psi\in \mathscr{F}^{\varepsilon_k}(\eta),
\end{align}
where $c$ does not depend on $k$.\\
Finally, for any $k\in \mathbb{N}$ and $f\in \rho^{\delta}L^2(L^{\varepsilon_k})$ we define
\begin{align*}
    F^{\varepsilon_k}(f):=\sum_{i=1}^{J(\nu)}\left(\int_{L^{\varepsilon_k}}f\cdot w_i^k\, d\vol_{L^{\varepsilon_k}}\right)e_i^k.
\end{align*}
Note that $e_i^k$ can be extended to a function on $L^{\varepsilon_k}$ supported in a small neighborhood of $\partial L^{\varepsilon_k}\smallsetminus(\mathscr{B}_{\varepsilon_k}\cup B_{2\eta}(p)\cup B_{2\eta}(q))$, and for which the H\"older and Sobolev norms are controlled by the ones of $e_i^k$ (with constants independent of $k$). These induce extensions of $F^{\varepsilon_k}(f)$ to $L^{\varepsilon_k}$, which we will still denote by $F^{\varepsilon_k}(f)$.

By \eqref{eq: uniform-estimate-eik}, there holds
\begin{align}
    \tnorm{F^{\varepsilon_k}(f)}_{2, \alpha; \nu} &\leq c \sum_{j = 1}^{J(\nu)}\Big|\int_{L^{\varepsilon_k}} f w_j^k \, d\vol_{L^{\varepsilon_k}} \Big|,  \label{Fest} \\
    \|F^{\varepsilon_k}(f)\|_{2;\delta } &\leq c \sum_{j = 1}^{J(\nu)}\Big|\int_{L^{\varepsilon_k}} f w_j^k \, d\vol_{L^{\varepsilon_k}} \Big|, \label{Fest2}
\end{align}
where $c$ is a constant independent of $k$. Similarly, $\Pi_\nu^{\varepsilon_k}(\psi)$ can be extended to a function on $L^{\varepsilon_k}$ supported in a small neighborhood of $\partial L^{\varepsilon_k}\smallsetminus(\mathscr{B}_{\varepsilon_k}\cup B_{2\eta}(p)\cup B_{2\eta}(q))$, and for which the H\"older and Sobolev norms are controlled by the ones of $\psi$. We will denote this extension by $\Pi_\nu^{\varepsilon_k}(\psi)$ also.
\begin{remark}
    For \(\varepsilon=0\), the image of the projection defined through \(B_\nu^0\) can be chosen to be the same as the image of $\Pi_{\nu}$ defined in \eqref{eq: def-projection-Pi}. Thus, the two projections agree up to
the choice of a finite-dimensional complementary subspace. Indeed, by the choice of $\delta(\nu)$, when the cones are disjoint one can choose the basis of $K_H^{C^1\cup C^2}(-\delta,\Delta)$ to be the one given (up to scalar multiplication) by the functions $(r^{\gamma_j}-r^{\gamma_j^-})\phi_j$ for $j\in \mathcal{I}_{\nu}$ (see Appendix \ref{ssec: existence-uniqueness}). Then \(\ker B_\nu^0\) is the \(L^2\)-orthogonal complement of the space spanned by the eigenfunctions \(\phi_j\) for $j\in \mathcal{I}_{\nu}$ on each link.
\end{remark}

Let $\varepsilon$ be an element of the sequence $\{\varepsilon_k\}_{k\in \mathbb{N}}$ introduced above. Set
    \begin{align*}
        \mathscr{F}_\ast^\varepsilon(\eta):=\{\psi\in \mathscr{F}^\varepsilon(\eta):\,\lVert \psi\rVert_{C^{2,\alpha}(\partial L^\varepsilon)}\leq \varepsilon^{\frac{m+5}{2}}\}.
    \end{align*}
    Let $\psi\in \mathscr{F}^\varepsilon_\ast(\eta)$.
    For any $u\in C^{2,\alpha,\nu}(L^\varepsilon)$, let $v\in C^{2,\alpha,\nu}(L^\varepsilon)$ be the unique solution of the Dirichlet problem \eqref{dp} for $\Delta$ on $L^\varepsilon$ with boundary data $\psi$, granted by \cite[Theorem 2.3]{Sm93}. Denote by $\mathcal{P}_{\psi}: C^{2,\alpha, \nu}(L^\varepsilon)\to C^{2,\alpha, \nu}(L^\varepsilon)$ the map sending $u$ to $v$. 
\begin{proposition}\label{prop: fixed-point-theorem}
    The map $\mathcal{P}_{\psi}$ sends $\mathscr{K}(\varepsilon, t)$ into $\mathscr{K}(\varepsilon,t)$ for $t$ small depending only on $m$ and $\alpha$. In particular, $\mathcal{P}_{\psi}$ has a fixed point $u \in \mathscr{K}(\varepsilon,t)$.
\end{proposition}
\begin{proof}
    We will prove that if $u \in \mathscr{K}(\varepsilon, t)$, then $v := \mathcal{P}_\psi(u) \in \mathscr{K}(\varepsilon,t)$ for $t$ sufficiently small. Write $v = v_1 + v_2$, where $v_1 \in C^{2,\alpha,\nu}(L^\varepsilon)$  solves 
    \begin{equation}
        \begin{cases}
    \Delta v_1 = f(u) - \Delta \mathcal{B}(u,\psi) &\text{ in } L^\varepsilon, \\
            v_1 = 0 &\text{ on } \partial L^\varepsilon,
        \end{cases} \label{v1}
    \end{equation}
    where $\mathcal{B}(u,\psi):=\Pi_{\nu}^\varepsilon(\psi)-F^\varepsilon(f(u))$ (recall that $f(u)=-\Theta_0^\varepsilon-R(u)$) and
    \begin{align*}
        v_2 := \mathcal{B}(u,\psi). \label{v2}
    \end{align*}
 We break the proof into five steps. \\
 
    \noindent
    \textbf{Step 1:} \emph{Preliminary estimates and estimates for $v_2$}. \\
    
    We first need to estimate $\|R(u)\|_{0;\delta}$ with weight $\delta = \nu + \frac{m}{2}-1$. Applying \eqref{bridgeest1} and \eqref{Fest} will then give the desired estimates for $v_2$. Using Proposition \ref{prop: pointerror}(a) and the definition of $\mathscr{K}(\varepsilon,t)$, we see that 
    \begin{align*}
        \|R(u)\|_{0;\delta}^2 &\leq c\int_{L^\varepsilon} |\nabla du|^4 \rho^{-2\delta} \, d\vol_{L^\varepsilon} +c\int_{L^\varepsilon}\lvert du\rvert^4\rho^{-2\delta-4 }\,d\vol_{L^\varepsilon}\\
        &= c\int_{L^\varepsilon} |\rho \nabla  du|^2 |\nabla  du|^2 \rho^{-2\delta - 2} \, d\vol_{L^\varepsilon} +c\int_{L^\varepsilon}\lvert du\rvert^2\lvert du\rvert^2\rho^{-2\delta-4}d\vol_{L^\varepsilon}\\
        &\leq c\tnorm{du}_{1;\nu-1}^2 \int_{L^\varepsilon} |\nabla  du|^2 \rho^{(2\nu - 4) - 2\delta} \, d\vol_{L^\varepsilon} +c\tnorm{du}_{1;\nu-1}^2\int_{L^\varepsilon}\lvert du\rvert^2\rho^{-2\delta-2}\rho^{2(\nu-2)}d\vol_{L^\varepsilon}\\
        &\leq c\tnorm{du}_{1;\nu-1}^2 \|\nabla du\|_{0;\delta}^2 +c\tnorm{du}_{1;\nu-1}^2 \| du\|_{0;\delta+1}^2\\
        &\leq c \varepsilon^{4  - 2t} \varepsilon^{m + 2 - 2t} = c \varepsilon^{m+6 - 4t},
    \end{align*}
    where we have applied $\nu > 2$ in the fourth line. We therefore have 
    \begin{equation}
        \|R(u)\|_{0;\delta} \leq c \varepsilon^{\frac{m+6}{2}- 2t}.\label{est1}   
    \end{equation}
    Using \eqref{Fest}, the Cauchy-Schwarz inequality, the triangle inequality, and \eqref{slagest} shows
    \begin{align}
        \tnorm{F^\varepsilon(f(u))}_{2, \alpha; \nu} &\leq c\sum_{j = 1}^{J(\nu)}\Big|\int_{L^\varepsilon} f(u)\rho^{-\delta} \cdot w_j \rho^{\delta} \, d\vol_{L^\varepsilon} \Big| \nonumber \\
        &\leq  c\big(\|\Theta_0\|_{0; \delta} + \|R(u)\|_{0;\delta}\big) \nonumber\\
        &\leq c\big(\varepsilon^{\frac{m+5}{2} } + \varepsilon^{\frac{m+6}{2}  - 2t} \big). \label{est01}
    \end{align}
    A similar estimate holds for $\|F^{\varepsilon}(f(u))\|_{2;\delta }$.
    By \ref{eq: holder-bound-bdry}, \eqref{Fest}, \eqref{Fest2}, \eqref{est1}, \eqref{est01},  the $L^2$ bound in \eqref{bridgeest1} we have in fact proved
    \begin{equation}
        \tnorm{dv_2}_{1, \alpha; \nu-1} \leq c \varepsilon^{\frac{m+5}{2}  } \text{ and } \lVert dv_2\rVert_{1; \delta+1} \leq c \varepsilon^{\frac{m+5}{2} }. \label{dv2}
    \end{equation}
    In particular, \eqref{dv2} implies $v_2 \in \mathscr{K}(\varepsilon,t)$ for any $t$ small.

    Notice that the right-hand side of \eqref{v1} is bounded in $\rho^{\delta} L^{2}(L^\varepsilon)$ by $c \varepsilon^{\frac{m+5}{2}}$ due to \eqref{est1}, the $L^2$ estimate in \eqref{bridgeest1}, and \eqref{eq: holder-bound-bdry}. We can therefore apply \cite[Lemma 4.1]{Sm93} to conclude 
    \begin{equation}
        \|v_1\|_{0;\delta+2} \leq c\lVert f(u) - \Delta \mathcal{B}(u,\psi)\rVert_{0;\delta}\leq  c \varepsilon^{\frac{m+5}{2} } \text{ for } t < \frac{1}{2m}. \label{est2}
    \end{equation}
    Integrating by parts, using that $|\nabla \rho| \leq 1$ on $L^\varepsilon$, 
    \begin{align*}
        \int_{L^\varepsilon} | d v_1|^2 \, \rho^{-2(\delta+1)}d\vol_{L^\varepsilon} &\leq \Big| \int_{L^\varepsilon} \langle v_1, \Delta v_1 \rangle \, \rho^{-2(\delta+1)} d\vol_{L^\varepsilon} \Big| + \Big|\int_{L^\varepsilon} 2(\delta+1)\langle \rho^{-\delta-1}dv_1, \rho^{-\delta}v_1 d\rho\rangle \, d\vol_{L^\varepsilon} \Big| \\
        &\leq \|v_1\|_{0;\delta+2}\|f(u)-\Delta\mathcal{B}(u,\psi)\|_{0;\delta} + 2(\delta+1)\|v_1\|_{0;\delta}\|dv_1\|_{0;\delta+1}.
    \end{align*}
    By \eqref{est2} and the discussion preceding it, the first term on the right-hand side is bounded by $c\varepsilon^{m+5}$. For the second term, we have
    \[2(\delta+1)\|v_1\|_{0;\delta}\|dv_1\|_{0;\delta+1} \leq (2(\delta+1))^2\frac{\|v_1\|_{0;\delta}^2}{2} + \frac{\|dv_1\|_{0;\delta+1}^2}{2}.\]
    As $\lVert v_1\rVert_{0;\delta}\leq c \lVert v_1\rVert_{0;\delta+2}$,
    re-arranging terms and applying \eqref{est2} once more gives
    \begin{equation}
        \|dv_1\|_{0; \delta +1} \leq c \varepsilon^{\frac{m+5}{2}  } \text{ for } t <  \frac{1}{2m}. \label{est3}
    \end{equation}
    This shows in particular that $\lVert dv_1\rVert_{0;\delta+1}$ satisfies the zeroth order Sobolev condition in $\mathscr{K}(\varepsilon,t)$. \\
    
    \noindent
    \textbf{Step 2:} \emph{H\"older estimates for $dv_1$ on $L_{\frac{1}{4}}^\varepsilon$}. \\
    
    We start with a $C^0$ bound on $v_1$. Let $\phi: L^{\varepsilon} \rightarrow [0,1]$ be a smooth radial cut-off function on $L^\varepsilon$ which is equal to 1 on $L_{\frac{1}{4}}^\varepsilon$ and is equal to zero on $L^\varepsilon \setminus L_{\frac{1}{8}}^\varepsilon$. 
    Proposition \ref{c0} applied to $\phi v_1$ with $p > \frac{m}{2}$ gives
    \begin{align}
        \|v_1\|_{C^0(L_{\frac{1}{4}}^\varepsilon)} \leq c\big(\|f(u)\|_{L^p(L_{\frac{1}{8}}^\varepsilon)} + \|\Delta \mathcal{B}(u,\psi)\|_{C^0(L_{\frac{1}{8}}^\varepsilon)} + \|v_1\|_{L^p([\frac{1}{8},\frac{1}{4}])} + \|dv_1\|_{L^p([\frac{1}{8},\frac{1}{4}])} \big). \label{est5}
    \end{align}
We estimate the terms on the right-hand side individually. 

By the argument for \eqref{est2}, the second term is bounded by $c\varepsilon^{\frac{m+5}{2}}$. Recall that $f(u)=-\Theta_0-R(u)$.
    Taking $p = \frac{m}{2} + \tau$ ( where $\tau>0$ is chosen small depending on $m$) gives 
    \begin{equation*}
        \|\Theta_{0}\|_{L^p(L_{\frac{1}{8}}^\varepsilon)} \leq c \varepsilon^{5 - \frac{2}{m} - O(\tau)} \leq c\varepsilon^{5 - \frac{3}{m}}.
    \end{equation*}
   To estimate the $R(u)$ term, we use the definition of $\mathscr{K}(\varepsilon,t)$ and Proposition \ref{prop: pointerror}(a):
    \begin{align}
        \|R(u)\|_{L^p(L_{\frac{1}{8}}^\varepsilon)} &\leq c\Big(\int_{L_{\frac{1}{8}}^\varepsilon}|\nabla du|^{2p} d\vol_{L^\varepsilon}\Big)^{\frac{1}{p}}+c \Big(\int_{L_{\frac{1}{8}}^\varepsilon}| du|^{2p} d\vol_{L^\varepsilon}\Big)^{\frac{1}{p}}\nonumber  \\
        &\leq c\|\nabla du\|_{C^0(L_{\frac{1}{8}}^\varepsilon)}^{2- \frac{4}{m + 2\tau}} \|\nabla du\|_{L^2(L_{\frac{1}{8}}^\varepsilon)}^{\frac{4}{m+2\tau}}+c \|du\|_{C^0(L_{\frac{1}{8}}^\varepsilon)}^{2-\frac{4}{m+2\tau}}\lVert du\rVert_{L^2(L_{\frac{1}{8}}^\varepsilon)}^\frac{4}{m+2\tau}\nonumber \\
        &\leq c \varepsilon^{(2 - t)(2- \frac{4}{m + 2\tau}) + (\frac{m+2}{2} - t)\frac{4}{m+2\tau}}+c\varepsilon^{(2-t)(2-\frac{4}{m+2\tau})+(\frac{m+2}{2}-t)\frac{4}{m+2\tau}} \nonumber \\
        &\leq c \varepsilon^{4-2t + \frac{2m-4}{m+2\tau}} \nonumber \\ 
        &\leq c\varepsilon^{6-\frac{4}{m} - O(t) - O(\tau)} \nonumber \\ 
        &\leq c \varepsilon^{5 - \frac{3}{m}} \label{est011}
    \end{align}
    for $t$ and $\tau$ small depending on $m$. Thus $\|f(u)\|_{L^p(L_{\frac{1}{8}}^\varepsilon)} \leq c \varepsilon^{5 - \frac{3}{m}}$.\\
    Next, we estimate the last term on the right-hand side of \eqref{est5}. From \eqref{est3}, if $m\geq 4$ we obtain
    \begin{equation}
        \|dv_1\|_{L^p([\frac{1}{8},\frac{1}{4}])} \leq \|dv_1\|_{C^0([\frac{1}{8},\frac{1}{4}])}^{1 - \frac{4}{m+2\tau}} \|dv_1\|_{L^2([\frac{1}{8},\frac{1}{4}])}^{\frac{4}{m+2\tau}} \leq c\|dv_1\|_{C^0([\frac{1}{8},\frac{1}{4}])}^{1 - \frac{4}{m+2\tau}}\varepsilon^{(\frac{m+5}{2} - 2t)\frac{4}{m+2\tau}}.
        \label{est7}
    \end{equation}
    Thus, we need to obtain a local $C^0$ estimate for $dv_1$. In doing so, we will also prove the desired estimate for the third term in \eqref{est5}. 
    
    Proposition \ref{prop: localest}(a) with $p = m$ gives
    \[\|v_1\|_{C^0([\frac{1}{16}, \frac{1}{3}])} \leq c\big(\|v_1\|_{L^2([\frac{1}{32}, \frac{2}{3}])} + \|f(u)-\Delta\mathcal{B}(u,\psi)\|_{L^m([\frac{1}{32}, \frac{2}{3}])}\big).\]
    Then \eqref{est2} and the same argument for \eqref{est011} with $p = m$ together imply
    \begin{equation}
        \|v_1\|_{C^0([\frac{1}{16}, \frac{1}{3}])} \leq c\varepsilon^{5-\frac{2}{m} - 2t}. \label{est05}
    \end{equation}
    Applying the local estimate Proposition \ref{prop: localest}(b) to $dv_1$ and Proposition \ref{prop: pointerror}(b) to $R(u)$, using that $u \in \mathscr{K}(\varepsilon,t)$, then yields
    \begin{equation}
        \|dv_1\|_{C^0([\frac{1}{8}, \frac{1}{4}])} \leq c\varepsilon^{4- \alpha - 2t}.\label{est03}
    \end{equation}
    Plugging \eqref{est03} into \eqref{est7} shows that the fourth term is smaller than $c\varepsilon^4$, while \eqref{est05} implies that the third term is bounded by $c\varepsilon^4$. It follows that for $m\geq 4$ 
    \begin{equation}
        \| v_1 \|_{C^0(L_{\frac{1}{4}}^\varepsilon)} \leq c\varepsilon^{4}. \label{est06}
    \end{equation}
     When $m = 3$ the estimate is simpler. From Proposition \ref{prop: globalschauder}, 
    \begin{align*}
         \| v_1 \|_{C^0(L_{\frac{1}{4}}^\varepsilon)} &\leq c\big(\|f(u)-\Delta \mathcal{B}(u,\psi)\|_{C^1(L_{\frac{1}{8}}^\varepsilon)} + \|v_1\|_{L^2([\frac{1}{8}, \frac{1}{4}])} + \|dv_1\|_{L^2([\frac{1}{8},\frac{1}{4}]}) \big). 
    \end{align*}
    Each of these terms is bounded by prior estimates, so that \eqref{est06} still hold for $m=3$.
    The global Schauder estimate Proposition \ref{prop: globalschauder}(a) applied to $\phi v_1$ gives us
    \begin{equation}
        \|dv_1\|_{C^{1}(L_{\frac{1}{4}}^\varepsilon)} \leq c\big(\varepsilon^{-2}\|v_1\|_{C^0(L_{\frac{1}{8}}^\varepsilon)} +  \|v_1\|_{C^{1}([\frac{1}{8}, \frac{1}{4}])} +\|f(u)-\Delta \mathcal{B}(u,\psi)\|_{C^{0,\alpha}(L_{\frac{1}{8}})}\big). \label{est010}
    \end{equation}
    Due to \eqref{est05} and \eqref{est06}, the first term is bounded by $c\varepsilon^{2}$. For the last term, \eqref{eq: holder-bound-bdry} and \eqref{Fest} imply
    \[\|f(u)-\Delta \mathcal{B}(u,\psi)\|_{C^{0,\alpha}(L_{\frac{1}{8}})} \leq \|\Theta_0\|_{C^{0,\alpha}(L_{\frac{1}{8}})} + \|R(u)\|_{C^{0,\alpha}(L_{\frac{1}{8}})}+\lVert \Theta_0\rVert_{0,\delta}+\lVert R(u)\rVert_{0,\delta}+\| \psi\|_{C^{2,\alpha}(\partial L^\varepsilon)}.\]
    The last term is bounded by $\varepsilon^{\frac{m+5}{2}}$.
    The second term on the right-hand side is estimated by $c\varepsilon^{4-2t-\alpha}$ using Proposition \ref{prop: pointerror}(b) and the estimates in the definition of $\mathscr{K}(\varepsilon,t)$, and the first term is estimated by $c\varepsilon^{2}$ by the $C^1$ bound on $\Theta_0$ in \eqref{bridgeest1}.
    The third and fourth terms are bounded by $c\varepsilon^{\frac{m+5}{2}}$ by \eqref{bridgeest1} and \eqref{est1}. 
    Finally, the term $\lVert v_1\rVert_{C^{1,\alpha}([\frac{1}{8}, \frac{1}{4}])} $ in \eqref{est010} is controlled by $c\varepsilon^{2}$ because of \eqref{est05}, \eqref{est03}, and Proposition \ref{prop: localest}(c) (where we use an estimate on $\lVert f(u)\rVert_{C^{0,\alpha}([\frac{1}{4},\frac{1}{2}])}$ as in the previous term). Combining each of the aforementioned cases, we conclude
    \begin{equation}
        \|dv_1\|_{C^{1}(L_{\frac{1}{4}}^\varepsilon)} \leq c \varepsilon^{2 } \label{1alpha}
    \end{equation}
    for $t$ small depending on $m$.
    Using Proposition \ref{prop: globalschauder}(b), we can estimate the $C^{1,\alpha}$ norm by 
    \begin{equation}\label{eq: upper-setimate-dv1}
        \| dv_1\|_{C^{1,\alpha}(L_{\frac{1}{4}}^\varepsilon)} \leq c\varepsilon^{2-\alpha}
    \end{equation}
    following the same argument as in the $C^1$ case. Thus, the target H\"older estimates for $dv_1$ on $L_{\frac{1}{4}}^\varepsilon$ are satisfied. It remains to prove the H\"older estimates for $dv_1$ at the cone tips, and the global Sobolev bound for $\nabla dv_1$. \\

    \noindent
    \textbf{Step 3:} \emph{H\"older estimates near the cone tips}. \\

    To estimate the H\"older norms of $dv_1$ close to the cone tips, we make use of Theorem \ref{thm: CHS-linear} with the truncated cone $C^1_\frac{1}{2}$, $C^2_\frac{1}{2}$ in place of $C_1$. In this setting, Theorem \ref{thm: CHS-linear} yields the estimate
    \begin{align*}
        \tnorm {v_1}_{2,\alpha;\nu}\leq c(\tnorm{R(u)}_{0,\alpha;\nu-2}+\lVert \Pi_\nu v_1\rVert_{C^{2,\alpha}(\partial L^\varepsilon_\frac{1}{2})}),
    \end{align*}
    as we may assume that $\mathcal{B}(u,\psi)$ is supported in $L^\varepsilon_\frac{3}{4}$. Here, $\Pi_\nu$ denotes the projection onto any component of $\partial(L^\varepsilon \setminus  L^\varepsilon_\frac{1}{2})$ (thought of as the link $\Sigma^i_\frac{1}{2}$ of a truncated cone $C^i_\frac{1}{2}$), onto the $L^2$-orthogonal complement of the space spanned by eigenfunctions of the Laplacian on $\Sigma_\frac{1}{4}^i$, whose corresponding positive indicial root (defined in \eqref{eq: pol-for-gamma}) is smaller than $\nu$.

    Now, by Proposition \ref{prop: pointerror}, for any $\sigma\in (0,\frac{1}{2}]$ and each $i=1,2$, we have
    \begin{align*}
        [R(u)]^i_{0,\alpha;\nu-2;\sigma}\leq c \sigma^{-2+\nu}[du]^i_{1;0;\sigma}[du]^i_{1,\alpha;0;\sigma}\leq c\varepsilon^{4-\alpha-2t}.
    \end{align*}
    On the other hand, arguing just as in \eqref{eq: holder-bound-bdry} we have that
    \begin{align*}
        \lVert \Pi_\nu v_1\rVert_{C^{2,\alpha}(\partial(L^\varepsilon \setminus  L^\varepsilon_\frac{1}{2}))}\leq c\lVert v_1\rVert_{C^{2,\alpha}(\partial(L^\varepsilon \setminus  L^\varepsilon_\frac{1}{2}))}\leq c\varepsilon^{2-\alpha},
    \end{align*}
    where the second inequality follows from \eqref{est06} and \eqref{eq: upper-setimate-dv1}. We therefore have 
    \begin{align*}
        \tnorm{v_1}_{2,\alpha;\nu}\leq c \varepsilon^{2-\alpha}.
    \end{align*}
    Applying Theorem \ref{thm: CHS-linear} with $\beta:=\frac{t}{2}$ in place of $\alpha$ and using that $\lVert v_1\rVert_{C^{2,\beta}(\partial(L^\varepsilon \setminus L^\varepsilon_\frac{1}{2}))}\leq c\varepsilon^{2-\beta}$ by interpolation of \eqref{1alpha} and \eqref{eq: upper-setimate-dv1}, we also get
    \begin{align*}
        \tnorm{v_1}_{2;\nu}\leq c \varepsilon^{2-\frac{t}{2}}.
    \end{align*}
    Therefore, $v_1$ satisfies the H\"older estimates in the definition of $\mathscr{K}(\varepsilon,t)$.\\

    \noindent
    \textbf{Step 4:} \emph{Global Sobolev estimate for $\nabla d v_1$}.\\

    The global Sobolev estimate for $\nabla dv_1$ follows \cite[Lemma 4.3]{Sm93}. For any $s > 0$, let $\Omega_s$ be the $s$-neighborhood of $\mathscr{B}_\varepsilon = \cup \mathscr{B}_{\varepsilon}^l$ in $L^\varepsilon$:
\[\Omega_s := \{x \in L^\varepsilon: d(x,\mathscr{B}_\varepsilon) < s\}.\]
Let $\phi$ be a cut-off function satisfying $\phi = 1$ on $L^\varepsilon \setminus \Omega_{20\varepsilon}$, $\phi = 0$ on $\Omega_{10\varepsilon}$, and $|\nabla^k\phi|\leq c\varepsilon^{-k}$ for $k = 1,2$.
Recall that
\begin{align*}
    \Delta: \rho^{\delta+2} H^2\cap H^1_0(C^1\cup C^2)\to \rho^{\delta} L^2(C^1\cup C^2)
\end{align*}
is Fredholm by \cite[Theorem 2.1]{Sm93}, and its kernel is empty since $\nu>\nu_0$. Therefore we have the estimate
\begin{align*}
    \|v_1\|_{2;\delta+2; L^\varepsilon \setminus \Omega_{20\varepsilon}} &\leq c\|\Delta(\phi v_1)\|_{0;\delta } \\
    &\leq c \big(\|f(u)-\Delta\mathcal{B}(u,\psi)\|_{0;\delta} + \varepsilon^{\frac{m}{2}-1}\|dv_1\|_{C^0(L_{\frac{1}{4}}^\varepsilon)} + \varepsilon^{\frac{m}{2}-2}\|v_1\|_{C^0(L_{\frac{1}{4}}^\varepsilon)}\big).
\end{align*}
The first term is bounded by $c \varepsilon^{\frac{m+5}{2}}$ due to \eqref{est1}, \eqref{bridgeest1}, and \eqref{eq: holder-bound-bdry}. The second and third terms are bounded by $c\varepsilon^{\frac{m+2}{2}}$ and $c\varepsilon^{\frac{m+4}{2}}$, respectively, due to \eqref{1alpha} and \eqref{est06}. Thus,
\begin{equation}
     \|v_1\|_{2;\delta+2; L^\varepsilon \setminus \Omega_{20\varepsilon}} \leq c \varepsilon^{\frac{m+2}{2}}. \label{globalsobolev1}
\end{equation}
In addition, \eqref{1alpha} implies 
\begin{equation}
\lVert dv_1\rVert_{L^2(\Omega_{20\varepsilon})}+\lVert \nabla dv_1\rVert_{L^2(\Omega_{20\varepsilon})}\leq  c\varepsilon^2\vol(\Omega_{20\varepsilon})^{\frac{1}{2}} \leq c\varepsilon^{\frac{m+3}{2}}. \label{globalsobolev2}
\end{equation}
Together, \eqref{globalsobolev1}, \eqref{globalsobolev2}, and \eqref{est3} imply
\begin{equation}
    \|dv_1\|_{1;\delta+1} \leq c \varepsilon^{\frac{m+2}{2}}. \label{globalsobolev}
\end{equation}
Combining \eqref{globalsobolev} with the estimates in Step 1--Step 3, we see that choosing $t$ and $\tau$ small depending on $m$, and then choosing $\varepsilon$ sufficiently small, we can ensure that if $u\in \mathscr{K}(\varepsilon,t)$, then $v = v_1 + v_2 \in \mathscr{K}(\varepsilon,t)$.\\

\noindent
    \textbf{Step 5:} \emph{Applying the Schauder Fixed Point Theorem.}\\

Let $\alpha'\in (0,\alpha)$ and let $\nu'\in (\nu_0,\nu)$.
By \cite[Theorem 2.3]{Sm93} and Proposition \ref{prop: contraction-estimate}, $\mathcal{P}_\psi\vert_{\mathscr{K}(\varepsilon, t)}$ is continuous with respect to the $C^{2,\alpha', \nu'}(L^\varepsilon)$ topology. We also observe that $\mathscr{K}(\varepsilon,t)$ is compact as a subset of $C^{2,\alpha', \nu'}(L^\varepsilon)$. In fact, the identity map $\iota: C^{2,\alpha, \nu}(L^\varepsilon)\to C^{2,\alpha', \nu'}(L^\varepsilon)$ is compact; thus, since $\mathscr{K}(\varepsilon, t)$ is a bounded subset of $C^{2,\alpha, \nu}(L^\varepsilon)$, it is precompact in $C^{2,\alpha', \nu'}(L^\varepsilon)$. Moreover, $\mathscr{K}(\varepsilon, t)\subset C^{2,\alpha', \nu'}(L^\varepsilon)$ is closed, as the norms appearing in the definition of $\mathscr{K}(\varepsilon, t)$ are lower semi-continuous with respect to $C^{2,\alpha', \nu'}$-convergence. We also note that $\mathscr{K}(\varepsilon, t)$ is convex, as it is defined by norm inequalities.
By the preceding steps, $\mathcal{P}_\psi$ maps $\mathscr{K}(\varepsilon, t)$, regarded as a compact convex subset of $C^{2,\alpha', \nu'}(L^\varepsilon)$, to itself. Therefore, by the Schauder Fixed Point Theorem, $\mathcal{P}_\psi$ must have a fixed point in $\mathscr{K}(\varepsilon,t)$.
\end{proof}

\begin{remark}
    Note that, as the sequence $\{\varepsilon_k\}_{k\in \mathbb{N}}$ introduced at the beginning of Subsection \ref{ssec: fixed-point-argument2} tends to zero, the parameter
    $\varepsilon>0$ in the proof of Proposition \ref{prop: fixed-point-theorem} can be chosen arbitrarily small. Since the fixed point $u$ produced by the proof satisfies $\tnorm{du}_{1,\alpha;\nu-1}\leq c \varepsilon^{2-\alpha-t}$ (for $t$ depending on $m$), this completes the proof of Theorem \ref{bridgeprinciple}.
    \end{remark}
    \begin{remark}\label{rmk: more-general-examples}
        Instead of two cones joined by $\varepsilon$-bridges, one can consider the following configuration: let $C^1$ and $C^2$ be truncated special Lagrangian cones, both tangent to a Lagrangian plane $T$ along a ray. Let $D$ be a compact smooth $m$-dimensional domain in $T$. Then one can construct $\varepsilon$-bridges joining $C^1$ to $D$ and $C^2$ to $D$ respectively, just as in Section \ref{section: approxsoln}, except that the bridges attach directly to D, without requiring any patching there. Note that the arguments in the proofs of Lemmas 4.1 and 4.2 in \cite{Sm93} remain true in this setting, since the kernel of $\Delta$ on $H_0^1(D)$ is trivial. As $D$ also has the same Lagrangian angle as the cones $C^i$, one can follow verbatim the proof of Proposition \ref{prop: fixed-point-theorem} in this setting. 
        
        A corresponding result can be obtained starting from a list of truncated special Lagrangian cones as in Theorem \ref{thm: main-theorem}.
        Note also that a similar result holds if some of the cones $C^1,...,C^N$ are replaced by compact domains with smooth boundary in smooth special Lagrangian submanifolds (so that they admit special Lagrangian extensions across their boundaries) provided the following condition is satisfied: for two consecutive objects $C^i$, $C^{i+1}$, there exists a plane $T^i$ such that $C^{i}$ and $C^{i+1}$ are tangent to $T^i$ along straight segments $\gamma_i^-\subset C^i$ and $\gamma_{i+1}^+\subset C^{i+1}$, respectively, with each segment meeting the boundary of the corresponding component orthogonally. In fact, under these assumptions, one can repeat the construction of the patching regions in Section \ref{section: approxsoln}, ensuring that the conditions \eqref{bridgeest1} are satisfied.
    \end{remark}
        
    Using Remark \ref{rmk: more-general-examples}, we may provide a simple geometric description of our solutions $L_{du}^\varepsilon$. Indeed, the solution special Lagrangians look like the cones glued together by planar special Lagrangian regions which are adjoined to the cone boundaries smoothly by thin ``almost flat" Lagrangian necks. It would be interesting to know if the neck regions of the solutions can be fattened post gluing to produce more natural geometric and analytic objects.

\subsection{Special Lagrangian graphs}
To conclude, we explain how the bridge construction can be carried out while preserving graphicality, thereby proving Corollary \ref{cor: SLEsolution}. Let $C=C(\Sigma)\subset\mathbb C^m$ be a regular special Lagrangian cone which is a gradient graph over the real factor:
\[
C=\operatorname{graph}(\nabla u_C)
\subset\mathbb R^m\times\mathbb R^m,
\]
where $u_C:\mathbb R^m\to\mathbb R$ is two-homogeneous and smooth away from the origin. Let $C_1=C_1(\Sigma)$ denote its unit truncation, and let
\[
\pi:\mathbb C^m=\mathbb R^m\times\mathbb R^m\longrightarrow\mathbb R^m
\]
be the projection onto the first factor.

Let $p\in \Sigma$ and set $T:=T_pC$.
Note that $T$ is itself a Lagrangian graph over $\mathbb{R}^m$. We first describe the construction for two copies of $C_1$. Let $v\in T$ be a unit vector orthogonal to $p$ and let $a=rv$ for any $r$ sufficiently large so that $\pi(C_1)\cap \pi(a+C_1)=\emptyset$.
Set $C_2:=a+C_1$. Since $a\in T$, $C_2$ is tangent to $T$ (as an affine plane) at $p+a$.

Next, we apply the flattening construction of Subsection \ref{section: approxsoln} near the endpoints $p\in C_1$ and $p+a\in C_2$: extend each cone radially around these points by a sufficiently small amount and modify it in an $O(\varepsilon)$-neighborhood of the rays passing through $p$ and $a+p$ . The resulting extension is Lagrangian and agrees with the plane $T$ along the outer part of its boundary. Note that since $T$ is graphical and the two cones are sufficiently far apart, by choosing $\varepsilon$ small enough we can ensure that the union of the extensions is again graphical over $\mathbb{R}^m$ (here we use the estimates on the patching regions obtained in Section \ref{section: approxsoln}). Note that one can connect the two outer parts of the boundary by a strip in $T$ of width $O(\varepsilon)$ (a ``racetrack"), which can be chosen so that the union $L^\varepsilon$ of the extended cones and the strip is a smooth Lagrangian submanifold away from the cone tips, graphical over a simply connected domain $\Omega_\varepsilon\subset \mathbb R^m$. That is, there is a function $q_\varepsilon: \Omega_\varepsilon\to\mathbb R$ such that $L^\varepsilon=\operatorname{graph}(\nabla q_\varepsilon)$.
The Lagrangian-angle defect of $L^\varepsilon$ is supported entirely in the two patching regions. Consequently, the estimates \eqref{bridge1}, \eqref{bridge2}, and \eqref{bridgeest1} hold as in Subsection \ref{section: approxsoln}, and $L^\varepsilon$ satisfies the hypotheses of Theorem \ref{bridgeprinciple}. Applying the theorem for any admissible choice of boundary datum $\psi$, we obtain that, for $\varepsilon$ in a sequence going to zero, there exists a function $u^\varepsilon$ on $L^\varepsilon$ with the following property: if $Q_\varepsilon(x):=(x,\nabla q_\varepsilon(x))$ and $\iota_\varepsilon$ is defined by
\begin{align*}
    \iota_\varepsilon: \Omega_\varepsilon\to \mathbb C^m, \quad x\mapsto \Psi^\varepsilon(Q^\varepsilon(x), du^\varepsilon(Q^\varepsilon(x))),
\end{align*}
where $\Psi^\varepsilon$ is the Weinstein map, then the image $L_{du^\varepsilon}$ of $\iota_\varepsilon$ is a special Lagrangian submanifold. Set $P^\varepsilon:=\pi\circ \iota_\varepsilon$. We will show that, for $\varepsilon$ sufficiently small, $P_\varepsilon$ is injective and its differential is invertible. This will imply that $L_{du^\varepsilon}$ is graphical over $\mathbb{R}^m$.

Write $P^\varepsilon(x)=x+E^\varepsilon(x)$ and note that by the scaling properties of $\Psi^\varepsilon$, we have the pointwise bounds
\begin{align*}
    \lvert E^\varepsilon\rvert\leq c\lvert du^\varepsilon\rvert,\quad \lvert DE^\varepsilon\rvert\leq c(\lvert du^\varepsilon\rvert+\lvert\nabla du^\varepsilon\rvert)
\end{align*}
for $c$ independent of $\varepsilon$.
Since the fixed point construction gives $\tnorm{du^\varepsilon}_{1;\nu-1}\leq \varepsilon^{2-t}$, we have $\lVert DE^\varepsilon\rVert_{C^0(\Omega_\varepsilon)}\leq c\varepsilon^{2-t}$.
Now for $x,y\in \Omega_\varepsilon$, there exists a smooth curve $\gamma$ in $\Omega_\varepsilon$ such that $L(\gamma)\leq Q\lvert x-y\rvert$, for some $Q$ independent of $\varepsilon$ by the way the patching regions and the connecting strip have been constructed. Thus,
\begin{align*}
    \lvert E^\varepsilon(x)-E^\varepsilon(y)\rvert\leq\int_\gamma\lvert DE^\varepsilon\rvert ds\leq CQ\varepsilon^{2-t}\lvert x-y\rvert.
\end{align*}
Hence,
\begin{align*}
    \lvert P^\varepsilon(x)-P^\varepsilon(y)\rvert\geq (1-CQ\varepsilon^{2-t})\lvert x-y\rvert.
\end{align*}
Choosing $\varepsilon$ to be sufficiently small, we can therefore ensure that $P^\varepsilon$ is injective. Moreover, for $\varepsilon$ sufficiently small $\lVert DE^\varepsilon\rVert_{C^0(\Omega_\varepsilon)}\leq \frac{1}{2}$ so that $DP^\varepsilon=\operatorname{Id}+DE^\varepsilon$ is invertible.

Writing $L_{du^\varepsilon}=\operatorname{graph}(\nabla U ^\varepsilon)$, the weighted decay of $du^\varepsilon$ and the conical equivariance of $\Psi^\varepsilon$ imply that
\[
    \lvert D^j\bigl(U^\varepsilon-u_C\bigr)\rvert
    = O\bigl(r^{\nu-j}\bigr),
    \qquad j=1,2,
\]
near each cone tip. Since $\nu>2$ and $u_C$ is two-homogeneous and smooth
away from the origin (up to subtracting a constant), $U^\varepsilon$ extends across each tip as a
$C^{1,1}$ function. Moreover, $U^\varepsilon$ satisfies the special
Lagrangian equation classically away from the finitely many tips, therefore
$U^\varepsilon$ is a viscosity solution on $P^\varepsilon(\Omega_\varepsilon)$ (see Lemma 4.1 in \cite{BOS26}).
A similar argument holds for the translated cone.
This concludes the argument in the case of two copies of the graphical cones. The argument can be easily adapted to any finite number of copies of the original graphical cone.

\appendix
\section{Technical computations}
For the reader's convenience, we have included a proof of the Conical Weinstein Neighborhood Theorem and a proof of unique solvability for the linear conical problem for $\Delta$ on a truncated special Lagrangian cone $C_1$.
\subsection{Conical Weinstein neighborhood}
\begin{proposition}[Conical Weinstein neighborhood]\label{prop:weinstein-approximate-solutions}
There exist $\eta_0>0$, a neighborhood 
\[
\mathcal{U}_{\eta_0} := \{(x,\xi)\in T^\ast C_{(0,r_0)}:\ |\xi|_{g_{L^\varepsilon}}<\eta_0 \rho(x)\},
\]
an open neighborhood $U\subset \C^m\setminus\{0\}$ of $L^\varepsilon$, conical around the singularities, and a symplectomorphism $\Psi:(\mathcal{U}_{\varepsilon_0},\omega_{\mathrm{can}})\to (U,\omega)$
such that\begin{enumerate}
\item \(\Psi(x,0)=x\) for every \(x\in L^\varepsilon\);

\item On each conical region \(L^\varepsilon_i\), the map \(\Psi\) is conical:
if \(\delta^i_\lambda(p_i+z)=p_i+\lambda z\) denotes the dilation centered at
\(p_i\), and if
\begin{align*}
\widetilde\delta^i_\lambda(x,\xi)
=
\left(
\delta^i_\lambda x,\,
\lambda^2 \bigl(d(\delta^i_\lambda)^{-1}\bigr)^*\xi
\right)
\end{align*}
is the induced dilation on \(T^*L^\varepsilon_i\), then $\Psi\circ \widetilde \delta^i_\lambda
=
\delta^i_\lambda\circ \Psi$
whenever both sides are defined;

\item \(\Psi^*\omega=\omega_{\mathrm{can}}\) on \(\mathcal U_{\eta_0}\).
\end{enumerate}
\end{proposition}

\begin{proof} Using $g_{L^\varepsilon}$ to identify $T^\ast L^\varepsilon$ with $TL^\varepsilon$, define
\begin{align*}
    F(x,\xi)=x+J(\xi^\sharp).
\end{align*}
Choose $\eta>0$ such that $F$ is a diffeomorphism from 
\begin{align*}
    \mathcal{U}_\eta=\{(x,\xi):\lvert \xi\rvert_{g_{L^\varepsilon}}<\eta\rho(x)\}
\end{align*}
onto its image. Around the bridges, this is possible by the tubular neighborhood theorem. To see that this can be achieved also near the cone tips, note that it is enough to verify the statement on annuli around every cone tip, and then use the fact that $F\circ\tilde \delta^i_\lambda=\delta^i_\lambda\circ F.$\\
Set
\begin{align*}
    \omega_0:=\omega_{\mathrm{can}},\quad \omega_1:=F^\ast\omega_{\C^m},\quad \omega_t=(1-t)\omega_0+t\omega_1\text{ for }t\in [0,1],
\end{align*}
where $\omega_{\mathrm{can}}$ denote the canonical symplectic form on $T^\ast L^\varepsilon$.
Note that $\omega_0$ and $\omega_1$ coincide along the zero section of $T^\ast L^\varepsilon$. Shrinking $\eta$ if necessary, we can assume that all $\omega_t$ are symplectic on $\mathcal{U}_\eta$.
Let $\beta=\omega_1-\omega_0$. Then $d\beta=0$ and on the zero section $\beta=0$.
If $s_\tau(x,\xi)=(x,\tau\xi)$ and $R$ is the fiber Euler vector field ($R=\sum_i\xi_i\frac{\partial}{\partial \xi_i}$ in local coordinates), define
\begin{align*}
    \alpha=\int_0^1s_\tau^\ast(\iota_R\beta)\frac{d\tau}{\tau}.
\end{align*}
Then $d\alpha=\beta$. Note that $\alpha$ and $\beta$ are homogeneous in the conical regions. Indeed, $\omega_0$ satisfies
\begin{align*}
    (\widetilde\delta^i_\lambda)^\ast\omega_0=(\widetilde\delta^i_\lambda)^\ast\omega_\mathrm{can}=\lambda^2\omega_\mathrm{can},
\end{align*}
and since \(F\circ\widetilde\delta^i_\lambda
=\delta^i_\lambda\circ F\) and
\((\delta^i_\lambda)^*\omega_{\C^m}=\lambda^2\omega_{\C^m}\), we have  
\[
(\widetilde\delta^i_\lambda)^*\omega_1
=
(F\circ\widetilde\delta^i_\lambda)^*\omega_{\C^m}
=
(\delta^i_\lambda\circ F)^*\omega_{\C^m}
=
F^*((\delta^i_\lambda)^*\omega_{\C^m})
=
\lambda^2\omega_1.
\]
Therefore $(\widetilde\delta^i_\lambda)^\ast\beta=\lambda^2\beta$ on the conical regions.\\
For $\alpha$, note that the maps \(s_\tau\) commute
with \(\widetilde\delta^i_\lambda\), and that the fiber Euler field is invariant under
\(\widetilde\delta^i_\lambda\), i.e.
\((\widetilde\delta^i_\lambda)_*R=R\); consequently
\begin{align*}
    (\widetilde\delta^i_\lambda)^*\alpha
=\int_0^1 s_\tau^*
\left(\iota_R (\widetilde\delta^i_\lambda)^*\beta\right)\frac{d\tau}{\tau}
=\lambda^2\alpha.
\end{align*}
Define the vector field $Y_t$ by
\begin{align*}
    \iota_{Y_t}\omega_t=-\alpha.
\end{align*}
Note that $Y_t$ vanishes on the zero section. The homogeneity of $\omega_t$ and $\alpha$ implies that $(\widetilde\delta^i_\lambda)_*Y_t=Y_t$. Therefore, the flow $\varphi_t$ of $Y_t$ commutes with dilations whenever defined. We claim that by choosing $\eta$ small enough, we can ensure that $\varphi_t$ exist for all $t\in [0,1]$ on $\mathcal{U}_\eta$. On the conical regions, by homogeneity it is enough to prove the statement on an annulus around each cone tip. Define
\begin{align*}
    q(x,\xi):=\frac{\lvert \xi\rvert_{g_{L^\varepsilon}}}{\rho(x)}.
\end{align*}
Since \(Y_t=0\) on the zero
section, we have \(|Y_t|\le cq\). Also \(d(q^2)\) vanishes on the zero section, so \(|d(q^2)|\le cq\).
Thus, for any integral curve \(\gamma(t)\) of \(Y_t\),
\begin{align*}
    \frac{d}{dt}q^2(\gamma(t))
=d(q^2)_{\gamma(t)}(Y_t(\gamma(t)))
\le cq^2(\gamma(t)).
\end{align*}

Gronwall's inequality gives \(q(\gamma(t))\le e^{ct/2}q(\gamma(0))\). Hence, choosing $\eta$ small enough we can ensure that the flow is defined for all $t\in [0,1]$ on the conical regions. Similarly, around the bridges, fix a compact neighborhood of the zero section where $Y_t$ is defined. As $Y_t$ vanishes on the zero section, there exist $c$ with $\lvert Y_t(x,\xi)\rvert\leq c\lvert \xi\rvert$. Thus we can apply Gronwall's inequality to $\lvert \xi\rvert_{g_{L^\varepsilon}}^2$. This implies that, after choosing $\eta$ smaller if necessary, $\varphi_t$ is defined for all $t\in [0,1]$ in $\mathcal{U}_\eta$.\\
Finally,
\[
\frac{d}{dt}\varphi_t^*\omega_t
=
\varphi_t^*(\dot\omega_t+\mathcal L_{Y_t}\omega_t)
=
\varphi_t^*(d\alpha+d(\iota_{Y_t}\omega_t))=0.
\]
Therefore \(\varphi_1^*\omega_1=\omega_0\). Define
\[
\Psi:=F\circ\varphi_1.
\]
Then
\[
\Psi^*\omega_{\mathbb{C}^m}=\varphi_1^*F^*\omega_{\mathbb{C}^m}=\varphi_1^*\omega_1=\omega_0=\omega_\mathrm{can}.
\]
Moreover \(\Psi(x,0)=x\) as $Y_t$ vanishes on the zero section, and on the conical regions \(\Psi\) is equivariant
because both \(F\) and \(\varphi_1\) are. Taking
\(U:=\Psi(\mathcal U_{\eta_0})\) gives the required neighborhood.    
\end{proof}
\subsection{Existence and uniqueness for the linear conical problem}\label{ssec: existence-uniqueness}
Here, we recall a few facts about existence and uniqueness for the Dirichlet problem on cones (and conical manifolds). For a modern exposition of the subject, see \cite[Section 2]{FMM}.

Let $C_1$ be a truncated special Lagrangian cone in $\C^m$ with link $\Sigma$. Let $0=\mu_0\leq\mu_1\leq \mu_2\leq...$ denote the eigenvalues of the non-negative operator
\(-\Delta_\Sigma\), counted with multiplicity, and choose an orthonormal basis
\(\{\phi_j\}_{j \in \mathbb{N}}\) for \(L^2(\Sigma)\) satisfying
\[
-\Delta_\Sigma\phi_j=\mu_j\phi_j.
\]
For any $j\in \mathbb{N}_0$, let $\gamma_j$ and $\gamma_j^-$ be respectively the non-negative and negative roots of
    \begin{align}\label{eq: pol-for-gamma}
        \gamma^2+(m-2)\gamma-\mu_j=0.
    \end{align}
Denote by $\Lambda_\Sigma$ the set of all positive roots $\gamma_j$. Such roots $\gamma_j$ are sometimes called \emph{indicial roots}.

For $\nu\in \mathbb{R}$, let
\begin{align}\label{eq: def-projection-Pi}
    \Pi_\nu: L^2(\Sigma)\to L^2(\Sigma), \qquad u\mapsto u-\sum_{i\in \mathcal{J}_\nu}\langle u,\phi_i\rangle\phi_i,
\end{align}
where $\mathcal{J}_\nu$ is the finite set of integers $i$ such that $\gamma_i<\nu$.
\begin{theorem}[\cite{CHS}, Theorem 1.1 \& Corollary 1.2]\label{thm: CHS-linear}
Fix \(\nu>0\) such that $\nu \notin \Lambda_\Sigma$
and let \(\Pi_\nu\) denote the projection onto the modes with indicial roots
larger than \(\nu\). Let $f\in C^{0,\alpha,\nu-2}(C_1)$ and suppose $\psi\in C^{2,\alpha}(\Sigma)$ with
$\alpha\in (0,1)$.
Then there exists a unique solution \(u\in C^{2,\alpha,\nu}(C_{1})\) of
\begin{align}\label{eq:dirichlet-u}
\begin{cases}
    \Delta u=f & \text{in } C_{1},\\
    \Pi_\nu u=\Pi_\nu\psi & \text{on } \Sigma.
\end{cases}
\end{align}
Moreover, we have the estimate
\[
    \tnorm{u}_{2,\alpha;\nu}
    \leq
    c\left(
       \tnorm{f}_{0,\alpha;\nu-2}
        +
        \|\Pi_\nu\psi\|_{C^{2,\alpha}(\Sigma)}
    \right),
\]
where \(c\) depends only on \(\nu,\alpha\), and the cone. 
\end{theorem}
\begin{proof}
We follow the exposition in \cite[Section 3]{SN91}.
    Decompose $\psi$ as
    \begin{align*}
        \psi=\sum_{j=0}^\infty\psi_j\phi_j
    \end{align*}
    for constant coefficients $\psi_j$.
    
    If $u\in C^2_{\text{loc}}(C_1)$, then for any $r>0$ we have
    \begin{align*}
        u(r\cdot)=\sum_{j=0}^\infty a_j(r)\phi_j.
    \end{align*}
    Now, if $u\in C^{2,\alpha,\nu}(C_1)$ is a solution of $\Delta u=f$, then for any $j\in \mathbb{N}$ there holds
    \begin{align}\label{eq: ODE-for-aj}
        (r^2\partial_r^2+r(m-1)\partial_r-\mu_j)a_j(r)=r^2f_j(r)
    \end{align}
    where
    \begin{align*}
        f_j(r):=\int_\Sigma f(r,\cdot)\phi_j\text{ for all }j\in \mathbb{N}.
    \end{align*}
    The corresponding homogeneous ODE for $a_j$ has solutions $r^{\gamma_j}$ and $r^{\gamma_j^-}$, where $\gamma_j$ and $\gamma_j^-$ are defined in \eqref{eq: pol-for-gamma}.
    For any $j\in \mathbb{N}$ we set
    \begin{align*}
        u_j(r)=\begin{cases}
            r^{\gamma_j}\int_0^r\tau^{1-m-2\gamma_j}\int_0^\tau s^{m-1+\gamma_j}f_j(s)dsd\tau&\text{ for }j\leq J \\
            \psi_j r^{\gamma_j}+r^{\gamma_j}\int_r^1\tau^{1-m-2\gamma_j}\int_0^\tau s^{m-1+\gamma_j}f_j(s)dsd\tau&\text{ for }j> J,
        \end{cases}
    \end{align*}
    where $J$ is an integer such that $\gamma_j>\nu$ if and only if $j> J$. Note that for any $j\in \mathbb{N}$, $u_j$ solves the ODE \eqref{eq: ODE-for-aj}. Therefore,
    \begin{align}\label{eq: def-of-u-series}
        u:=\sum_{j=0}^\infty u_j\phi_j
    \end{align}
    formally solves \eqref{eq:dirichlet-u}. It will thus be sufficient to show that $u$ in \eqref{eq: def-of-u-series} satisfies the claimed estimate.
    
    We will first show that
    \begin{align}\label{eq: claim-l2-control}
        \lvert u(r\cdot)\rvert_{L^2(\Sigma)}\leq cr^\nu \left(\tnorm{f}_{0;\nu-2}+\lVert\Pi_\nu\psi\rVert_{L^2(\Sigma)}\right).
    \end{align}
    In fact, if $j\leq J$
    \begin{align}\label{eq: first-part-L2}
        \lvert u_j(r)\rvert\leq \left\lvert r^{\gamma_j}\int_0^r\tau^{1-m-2\gamma_j}\int_0^\tau s^{m-1+\gamma_j}s^{\nu-2}ds d\tau\right\rvert\tnorm{f}_{0;\nu-2}\leq c r^\nu\tnorm{f}_{0;\nu-2}. 
    \end{align}
    For $j> J$ we compute
    \begin{align}\label{eq: triangular-for-uj}
        \left(\sum_{j>J}\lvert u_j(r)\rvert^2\right)^\frac{1}{2}\leq \left(\sum_{j>J}\lvert \psi_j\rvert^2r^{2\gamma_j}\right)^\frac{1}{2}+\left(\sum_{j>J}\lvert b_j(r)\rvert^2\right)^\frac{1}{2},
    \end{align}
    where for any $j\in \mathbb{N}$
    \begin{align*}
        b_j(r)=r^{\gamma_j}\int_r^1\tau^{1-m-2\gamma_j}\int_0^\tau s^{m-1+\gamma_j}f_j(s)dsd\tau.
    \end{align*}
    For the first term on the right hand side of \eqref{eq: triangular-for-uj} we have
    \begin{align}\label{eq: boundary-estimate-for-L2}
    \left(\sum_{j>J}\lvert \psi_j\rvert^2r^{2\gamma_j}\right)^\frac{1}{2}\leq r^{\gamma_{J+1}}\lVert \Pi_\nu\psi\rVert_{L^2(\Sigma)}\leq r^\nu\lVert \Pi_\nu\psi\rVert_{L^2(\Sigma)},
    \end{align}
    as $\nu<\gamma_{J+1}$.
    For the second term we estimate we use Minkowski's integral inequality:
    \begin{align*}
        \left(\sum_{j>J}\lvert b_j(r)\rvert^2\right)^\frac{1}{2}\leq \int_r^1\int_0^\tau\left(\sum_{j>J}r^{2\gamma_j}\tau^{2-2m-4\gamma_j}s^{2m-2+2\gamma_j}f_j^2(s)\right)^\frac{1}{2}dsd\tau.
    \end{align*}
One verifies that for the values of $\tau$ and $s$ in the integral above, $r^{2\gamma}\tau^{2-2n-4\gamma}s^{2n-2+2\gamma}=\tau^{2-2n}s^{2n-2}\left(\frac{rs}{\tau^2}\right)^{2\gamma}$ is a non-increasing function of $\gamma$. Therefore, we have
\begin{align}\label{eq: third-part-L2}
        \left(\sum_{j>J}\lvert b_j(r)\rvert^2\right)^\frac{1}{2}\leq& \int_r^1\int_0^\tau r^{\gamma_{J+1}}\tau^{1-m-2\gamma_{J+1}}s^{m-1+\gamma_{J+1}} \left(\sum_{j>J}f_j^2(s)\right)^\frac{1}{2}dsd\tau\\
        \nonumber 
    \leq&r^{\gamma_{J+1}}\tnorm{f}_{0;\nu-2}\int_r^1\tau^{1-m-2\gamma_{J+1}}\int_0^\tau s^{n-1+\gamma_{J+1}+\nu-2}dsd\tau\\
    \nonumber
        \leq&c r^\nu\tnorm{f}_{0;\nu-2}.
    \end{align}
Combining \eqref{eq: first-part-L2},\eqref{eq: boundary-estimate-for-L2} and \eqref{eq: third-part-L2}, we see that \eqref{eq: claim-l2-control} holds.  

Fix $\sigma \in (0, 1/2)$. By the local interior supremum estimates, applied
after rescaling to fixed annuli (see \cite[Sections 9.7 \& 9.9]{GT01} as well as Proposition \ref{prop: localest}(a)), we find that on slightly smaller annuli \((\sigma_1, \sigma_2] \Subset (\sigma, 2\sigma]\)
\[
\|u\|_{C^0((\sigma_1, \sigma_2])}
    \leq
    c\left(
        \sigma^{-m/2}\|u\|_{L^2((\sigma, 2\sigma])}
        +
        \sigma^2\|f\|_{C^0((\sigma,2\sigma])}
    \right)
    \leq
    c \sigma^\nu
    \left(
        \tnorm{f}_{0;\nu-2}
        +
        \|\Pi_\nu\psi\|_{L^2(\Sigma)}
    \right).
\]
Applying the corresponding interior Schauder estimates, again
after rescaling to fixed annuli (see \cite[Sections 6.1 \& 6.2]{GT01} and Proposition \ref{prop: localest}(b)--(c)), we
obtain
\[
    \tnorm{u}_{2,\alpha;\nu}
    \leq
    c\left(
        \tnorm{f}_{0,\alpha;\nu-2}
        +
        \|\Pi_\nu\psi\|_{C^{2,\alpha}(\Sigma)}
    \right).
\]
This is the desired estimate. A similar argument holds for $\sigma = 1/2$ using boundary supremum and Schauder estimates in places of the interior estimates. 

Finally, uniqueness of the solution follows from the fact that --- as observed before --- for any solution $v\in C^{2,\alpha,\nu}(C)$ of \eqref{eq:dirichlet-u} and for any $j\in \mathbb{N}$, $v_j(r):=\int_\Sigma v\phi_j$ must satisfy three conditions simultaneously: (i.) solve the ODE \eqref{eq: ODE-for-aj}, (ii.) satisfy the estimate $\lvert v_j(r)\rvert\leq c r^\nu$, and (iii.) $v_j$ attain the boundary datum (unless $j\leq J$, in which case the unique solution satisfying the decay condition is given by $u_j$).
Now, (i.) implies that $v_j-u_j$ must be a linear combination of $r^{\gamma_j}$ and $r^{\gamma_j^-}$, and from (ii.), (iii.) we deduce that it must be equal to zero. But the only solutions of \eqref{eq: ODE-for-aj} different from $u_j$ must have a term involving $r^{\gamma_j^-}$, and therefore are not admissible.
\end{proof}

\begin{remark}\label{rmk: CHS}
    Theorem \ref{thm: CHS-linear} holds for $\Delta$ on $L^\varepsilon$ also, provided $\Pi_\nu$ is defined by \eqref{eq: epsilon-projection}. However, in this case we must assume $\nu \notin \cup_{i=1}^N \Lambda_{\Sigma^i} $. In fact, an analogous theorem holds for a broad class of submanifolds with isolated singularities that includes the conically singular submanifolds, as well for a large class of operators including the ``conic operators" --- for which $\Delta_{g_C}$ is the canonical example \cite[see Theorem 2.3]{Sm93}.
\end{remark}
\begin{coi*}
    The authors confirm that they have no conflict of interest regarding this work.
\end{coi*}
\printbibliography
\end{document}